\documentclass[reqno,a4paper]{amsart}
\usepackage[T1]{fontenc}

\usepackage{latexsym, amssymb, amsmath}
\usepackage{enumitem}
\usepackage{amsthm}
\usepackage{tikz-cd}
\usepackage{orcidlink}
\usepackage{doi}
\renewcommand{\doi}[1]{DOI:~\href{https://doi.org/#1}{#1}}

\usepackage[utf8]{inputenc}
\usepackage[british]{babel}
\usepackage[all]{xy}
\usepackage{calrsfs}
\usepackage{mathabx}
\usepackage{xcolor}
\usepackage{blkarray}
\usepackage[textsize=scriptsize]{todonotes}
\usepackage{thmtools,mathtools}
\usepackage{nameref,hyperref,cleveref}
\usepackage{quiver}
\usepackage{verbatim}

\theoremstyle{plain}
\newtheorem{theorem}[subsection]{Theorem}
\newtheorem{proposition}[subsection]{Proposition}
\newtheorem{corollary}[subsection]{Corollary}	
\newtheorem{lemma}[subsection]{Lemma}

\theoremstyle{definition}
\newtheorem{definition}[subsection]{Definition}

\theoremstyle{remark}
\newtheorem{remark}[subsection]{Remark}
\newtheorem{example}[subsection]{Example}
\newtheorem{examples}[subsection]{Examples}

\numberwithin{equation}{section}

\newcommand{\noproof}{\hfil\qed}

\newcommand{\0}{\emptyset}

\newcommand{\bC}{\mathsf{C}}

\newcommand{\bF}{\mathbb{F}}         
\newcommand{\bV}{\mathsf{V}}

\newcommand{\bZ}{\mathbb{Z}}
\newcommand{\bU}{\mathsf{U}}

\DeclareMathOperator{\Aut}{Aut}
\DeclareMathOperator{\op}{op}

\DeclareMathOperator{\id}{id}

\newcommand{\Hp}{\ensuremath{\mathsf{Hoops}}}
\newcommand{\WH}{\ensuremath{\mathsf{WHoops}}}
\newcommand{\PH}{\ensuremath{\mathsf{PHoops}}}
\newcommand{\MV}{\ensuremath{\mathsf{MVAlg}}}
\newcommand{\PA}{\ensuremath{\mathsf{PAlg}}}
\newcommand{\Gp}{\ensuremath{\mathsf{Grp}}}
\newcommand{\Set}{\ensuremath{\mathsf{Set}}}
\newcommand{\Cat}{\ensuremath{\mathsf{Cat}}}
\newcommand{\Rng}{\ensuremath{\mathsf{Rng}}}

\newcommand{\Pt}{\ensuremath{\mathsf{Pt}}}
\newcommand{\Ring}{\ensuremath{\mathsf{Ring}}}
\newcommand{\CRing}{\ensuremath{\mathsf{CRing}}}
\newcommand{\Assoc}{\ensuremath{\mathsf{Assoc}}}

\newcommand{\Lie}{\ensuremath{\mathsf{Lie}}}
\newcommand{\Leib}{\ensuremath{\mathsf{Leib}}}

\newcommand{\CAssoc}{\ensuremath{\mathsf{CAssoc}}}

\newcommand{\AbAlg}{\ensuremath{\mathsf{AbAlg}}}
\newcommand{\Alg}{\ensuremath{\mathsf{Alg}}}

\newcommand{\Alt}{\ensuremath{\mathsf{Alt}}}

\newcommand{\Vect}{\ensuremath{\mathsf{Vec}}}

\newdir{>}{{}*:(1,-.2)@^{>}*:(1,+.2)@_{>}}
\newdir{<}{{}*:(1,+.2)@^{<}*:(1,-.2)@_{<}}
\newdir{>>}{{}*!/3.5pt/:(1,-.2)@^{>}*!/3.5pt/:(1,+.2)@_{>}*!/7pt/:(1,-.2)@^{>}*!/7pt/:(1,+.2)@_{>}}
\newdir{ >}{{}*!/-8pt/@{>}}

\date{}

\newcommand{\w}{\rightarrow}

\begin{document}

\title[Coherent and ideal actions in ideally exact categories]{Coherent and ideal actions\\in ideally exact categories}

\author[M.~Mancini]{Manuel Mancini~\orcidlink{0000-0003-2142-6193}}
\author[G.~Metere]{Giuseppe Metere~\orcidlink{0000-0003-1839-3626}}
\author[F.~Piazza]{Federica Piazza~\orcidlink{0009-0001-1028-9659}}

\email{manuel.mancini@unipa.it; manuel.mancini@uclouvain.be}
\email{giuseppe.metere@unimi.it}
\email{federica.piazza07@unipa.it; federica.piazza1@studenti.unime.it}

\address[M.~Mancini, F.~Piazza]{Dipartimento di Matematica e Informatica, Università degli Studi di Palermo, Via Archirafi 34, 90123 Palermo, Italy.}

\address[M.~Mancini]{Institut de Recherche en Mathématique et Physique, Université catholique de Louvain, chemin du cyclotron 2 bte L7.01.02, B--1348 Louvain-la-Neuve, Belgium.}

\address[G.~Metere]{Dipartimento di Scienze per gli Alimenti, la Nutrizione e l'Ambiente, Università degli Studi di Milano Statale, Via Celoria 2, 20133 Milano, Italy.}

\address[F.~Piazza]{Dipartimento di Scienze Matematiche e Informatiche, Scienze Fisiche e Scienze della Terra, Università degli Studi di Messina, Viale Ferdinando Stagno d'Alcontres 31, 98166 Messina, Italy.}


\begin{abstract}
In the context of ideally exact categories, we introduce the notions of internal \emph{coherent} action and internal \emph{ideal} action that generalise different aspects of unital actions of rings and algebras. We prove that every ideal action is coherent, and that the converse statement holds in some relevant ideally exact contexts. Furthermore, a connection with G.~Janelidze’s notion of semidirect product in ideally exact categories is analysed.
\end{abstract}

\subjclass[2020]{03B52; 06D35; 08C05; 16B50; 18C05; 18E13}
\keywords{Ideally exact category, semi-abelian category, internal action, non-associative algebra, MV-algebra, product algebra}

\maketitle

\section{Introduction}\label{sec_intro}
The concept of action---or of external operation---is a pervasive algebraic notion, as it enables algebraic structures to interact beyond their intrinsic environment of definition. 

In categorical terms, such actions can be defined as functors. 
For instance, an action of a group $G$ on a set $X$ can be described as a functor
\[
\alpha\colon G(*) \to \Set,
\]
where $G(*)$ is the group $G$ viewed as a one-object groupoid, and the unique object~$*$ is sent to $X$ by $\alpha$.

However, it is sometimes necessary to consider \emph{internal actions}, where the acting and the acted-upon objects belong to the same category. These kinds of actions are useful tools for studying the algebraic structures themselves, as it happens systematically in cohomological algebra. In the case of groups, a functor $\alpha\colon G(*) \to \Gp$ corresponds to a group homomorphism $G \to \Aut(X)$, where $X=\alpha(*)$ and $\Aut(X)$ denotes the group of automorphisms of $X$. 

On the one hand, this approach offers certain advantages, as it enables the expression of the notion of action within the internal language of the category of groups. On the other hand, it also presents several drawbacks. Firstly, many algebraic categories, otherwise well-behaved, do not admit an internal construction that plays the role of $\Aut(-)$. Secondly, even in cases where such a construction exists, it is typically not functorial.

The right take on this subject was developed by D.~Bourn and G.~Janelidze in~\cite{Bourn-Janelidze:Semidirect}, and later systematised by F.~Borceux, G.~Janelidze and G.~M.~Kelly in~\cite{IntAct}. In these two papers, the authors make a strong point on what an internal action should be from a \emph{category-theoretic perspective}. The first paper centres on the algebraic (monadic) description of split epimorphisms, i.e., of the category of \emph{points} (see \Cref{sec_prel}). In a Barr-exact pointed protomodular category $\bC$ with pushouts of split monomorphisms (and thus in a semi-abelian category, since finite coproducts exist automatically in $\bC$), the split epimorphisms with a fixed codomain $B$ are given by algebras on the kernel of the epimorphism, for a specified monad $B\flat(-) \colon \bC \to \bC$. When the category is not pointed, the general theory introduced in~\cite{Bourn-Janelidze:Semidirect} still applies, albeit with a caveat: it only allows us to express split epimorphisms over $B$ in terms of split epimorphisms over another object $E$, whenever there is a morphism $E \to B$.

However, the ``working mathematician'' knows that proper notions of (external) actions have been defined in many non-pointed algebraic contexts, as for example for unital algebras over a field $\bF$ or for unital rings~\cite{maclane_homology}, and they are often obtained by imposing additional axioms to a corresponding notion of action in a pointed context (see, for instance,~\cite{Orzech1, Orzech2}, where external actions are described in the context of groups with operations). 

The present paper aims to partially bridge this gap, by proposing the new notions of \emph{coherent action} and of \emph{ideal action}. From an algebraic perspective, the former generalises unital actions of rings and algebras, where a multiplicative unit is required to \emph{act} like a unit, while the latter stems from the classical situation where a unital algebra acts on an ideal.

\smallskip
The manuscript is structured as follows. After this introduction, Section 2 presents some necessary background. \emph{Ideally exact categories} are recalled in Section 3, while in Section 4 we propose the notions of coherent and ideal actions, together with their morphisms. In Theorem~\ref{thmunital}, we show that all ideal actions are coherent; the problem of whether the converse of the theorem holds in general remains open. We call \emph{BAT} the ideally exact contexts with a \emph{good theory of actions}, i.e., where the converse of Theorem~\ref{thmunital} holds for objects and morphisms. A connection with semidirect products in the ideally exact context is analysed in Section 5, and equivalent conditions for an ideally exact context to be BAT are given (Corollary~\ref{cor:main}). 
The final section presents case studies of BAT contexts, namely: unital non-associative $\bF$-algebras or rings, MV-algebras, product algebras, and $\Set^{\op}$, the dual of the category of sets.

\medskip
\textbf{Notation.} We often identify algebraic varieties with their corresponding categories; therefore, we shall adopt the same notation for both.

\section{Preliminaries}\label{sec_prel}
Let us briefly recall that, for an object $B$ of a category $\bC$, one can define the slice category~$(\bC\downarrow B)$, having for objects the arrows with codomain $B$, and for morphisms the obvious commutative triangles. 
Let $\bC$ be a category with pullbacks, with initial object~$\0$\footnote{Notice that the symbol $\0$ will not be used to denote the empty set, except, of course, for the categories where the empty set is initial.} and terminal object $1$. From now on, we denote by $\iota_B$ the unique map $\0 \to B$, and by~$\tau_B$ the unique map $B \to 1$. When the category is \emph{pointed}, i.e., when the unique map $\iota_1=\tau_\0 \colon \0 \to 1$ is an isomorphism, we denote the initial/terminal object with $0$. 

We recall that the \emph{adjunction associated} with a morphism $f \colon E \to B$ of $\bC$ is
\[
\begin{tikzcd}
{(\bC\downarrow B)} & {(\bC \downarrow E),}
\arrow[""{name=0, anchor=center, inner sep=0}, "{f^{*}}"', from=1-1, to=1-2]
\arrow[""{name=1, anchor=center, inner sep=0}, "{f \circ -}"', curve={height=16pt}, from=1-2, to=1-1]
\arrow["\dashv"{anchor=center, rotate=-89}, draw=none, from=1, to=0]
\end{tikzcd}
\]
where $f^{*}$ is the pullback functor along $f$, determined by a fixed choice of pullbacks along~$f$. More precisely, for any object $p\colon A\to B$ of $(\bC\downarrow B)$, the object $f^*(p)$ of $(\bC\downarrow E)$ is the projection
\[
p_1 \colon E \times_B A \to E
\]
in the pullback diagram
\[
\begin{tikzcd}
{E\times_{B} A} & A \\
E & B.
\arrow["{p_2}", from=1-1, to=1-2]
\arrow["{p_1}"', from=1-1, to=2-1]
\arrow["p", from=1-2, to=2-2]
\arrow["f"', from=2-1, to=2-2]
\end{tikzcd}
\]
Notice that the unit and the counit of the adjunction $(f \circ -)\dashv f^{*}$ are \emph{cartesian}.

Another relevant construction on $\bC$ is the category $\Pt_{\bC}(B)$ of \emph{points over $B$}, i.e., the category of \emph{pointed objects} of $(\bC\downarrow B)$. Explicitly, objects are split epimorphisms over~$B$ with a chosen splitting, and morphisms are arrows between the domains of such split epimorphisms, commuting with both the retractions and the sections.

Let $f\colon E \to B$ be a morphism of $\bC$. If $\bC$ admits pullbacks along $f$, we can define the functor
\[
f^*\colon \Pt_{\bC}(B) \to \Pt_{\bC}(E)
\](notice that we are adopting the same notation here, as for slice categories, since the construction is essentially the same). Furthermore, if $\bC$ admits pushouts along~$f$, the functor $f^*$ has a right adjoint 
\[
f_!\colon \Pt_{\bC}(E) \to \Pt_{\bC}(B),
\]
defined by
\[
\begin{aligned}
\xymatrix{
D\ar@<+.5ex>[d]^{d}
\\
E\ar@<+.5ex>[u]^{e}
}    
\end{aligned}
\quad\longmapsto\quad
\begin{aligned}
\xymatrix{
D\ar@<+.5ex>@{.>}[d]^{d}\ar@{.>}[r]
&D+_EB\ar@<+.5ex>[d]^{[f\circ d, \ \id_B]}
\\
E\ar@<+.5ex>@{.>}[u]^{e}\ar@{.>}[r]_-{f}
&B,\ar@<+.5ex>[u]^{\iota_2}
}    
\end{aligned}
\]
where $f_!(e,d)=(\iota_2,[f\circ d, \id_B])$ is given by the universal property of the pushout of $e$ along $f$.

From now on, let us tacitly suppose that $\bC$ has pullbacks and pushouts, so that~$f^*$ and~$f_!$ are defined for all morphisms $f \colon E \to B$ of $\bC$. Let us fix some standard terminology. 

The category $\bC$ is \emph{protomodular}~\cite{protomodular} when, for every morphism $f \colon B \to E$ in it, the functor $f^*$ on points reflects isomorphisms. Notice that if $\bC$ is \emph{pointed}, protomodularity is equivalent to the \emph{split short five-lemma}~\cite{reg_prot}. $\bC$ is \emph{regular}~\cite{Barr} if it is a finitely complete category and all effective equivalence relations have pullback stable coequalizers. A regular category $\bC$ is \emph{Barr-exact}~\cite{Barr} if all equivalence relations are effective, i.e., kernel pairs. Finally, a category $\bC$ is \emph{semi-abelian}~\cite{Semi-Ab} when it is pointed, with finite coproducts, protomodular and Barr-exact. 

Let $\bC$ be a protomodular category. If the functors $f^*$ are not only conservative, but monadic, $\bC$ is said to be a category with \emph{semidirect products}~\cite{Bourn-Janelidze:Semidirect}. Indeed, if we denote by~$T^f$ the monad determined by $f^*$, the $T^f$-algebras are called \emph{internal actions}. One defines the semidirect product $(X,\xi) \rtimes (B,f)$ of $(B,f)$ with a $T^f$-algebra $(X,\xi)$ as a preimage of $(X,\xi)$ along the comparison equivalence $K$:
\[
\xymatrix@!C=14ex{
&(X,\xi)\ \ \in \ar@{|-->}[dl]
&\Pt_{\bC}(E)^{T^f}\ar[d]^{\text{forgetful}}
\\(X,\xi)\rtimes(B,f)\ \ \in 
&\ar@<+1ex>[ur]^K\Pt_{\bC}(B)\ar[r]_-{f^*}&\Pt_{\bC}(E). \ar@/_2.5ex/[l]^{\bot}_(.35){f_!}}
\]
When $\bC$ is pointed, one may take $E=0$, so that $\Pt_{\bC}(E)\cong \bC$. In this case, $f^*=\iota_B^*$ is nothing but the kernel functor, so that we reproduce a more familiar notion of semidirect product. Notice that in this case, we slightly modify our notation: $B\flat \coloneqq T^{\iota_B}$, where the object $B\flat X$ is given by a kernel
\[
\kappa_{B,X} \colon B\flat X \to B+X
\]
of the morphism $[\id_B,0]\colon B+X \to B$. Hence, internal actions are morphisms of the form
\[
\xi \colon B \flat X \to X
\]
and monadicity establishes that
\begin{equation}\label{equivalence}
K \colon \Pt_\bC(B) \to \bC^{B \flat}.
\end{equation}
is an equivalence of categories between points over $B$ and internal actions of $B$. Notice that $(-) \flat (-)$ is functorial on both components.

\section{Ideally exact categories}\label{ideallyexactcat}
Ideally exact categories have been introduced by G.~Janelidze in~\cite{IdeallyExact} as a non-pointed counterpart of semi-abelian categories. In fact, ideally exact categories are closely related to a categorical setting introduced by S.~Lapenta, L.~Spada and the second named author in~\cite{rel} in order to define a \emph{categorical notion of ideal} relative to a given adjunction $U\vdash F$ (a \emph{basic setting for relative $U$-ideal}, see~\cite[Definition~3.3]{rel}). In the present article, we stick to G.~Janelidze's definition, since, among many of its features, it clarifies how the notion of relative $U$-ideal can be defined intrinsically.

\begin{definition}\cite[Definition 3.2 and Theorem 3.1]{IdeallyExact}
A category $\bU$ is \emph{ideally exact} if it is Barr-exact, protomodular, has finite coproducts, and the unique morphism $\0 \to 1$ in $\bU$ is a regular epimorphism.
\end{definition}

In addition to all semi-abelian categories, examples of ideally exact categories include the categories $\Ring$ and $\CRing$ of unital and commutative unital rings respectively, every category of unital algebras over a field $\bF$ (see \Cref{Var}), the categories of \emph{MV-algebras}, \emph{product algebras} (see~\cite{rel} and \Cref{MV}), and any cotopos (see \Cref{sec_set}).

Since the pullback functor $\bU \to (\bU \downarrow \0)$ along the regular epimorphism $\0 \to 1$ is monadic, ideal exactness can be related to adjoint pairs.
 
\begin{theorem}\cite[Theorem 3.1]{IdeallyExact}\label{carid}
Let $\bU$ be a category with pullbacks. The following conditions are equivalent:
\begin{enumerate}
\item $\bU$ is ideally exact;
\item $\bU$ is Barr-exact, has finite coproducts and there exists a monadic functor $\bU \to \bV$, where $\bV$ is a semi-abelian category;
\item There exists a monadic functor $\bU \to \bV$, where $\bV$ is a semi-abelian category, such that the underlying functor of the corresponding monad preserves regular epimorphisms and kernel pairs. \noproof
\end{enumerate}
\end{theorem}

\begin{remark}\cite[Theorem 3.3]{IdeallyExact}\label{cartesian}
Let $\bU$ be an ideally exact category. It follows from Theorem~\ref{carid} that there exists a monadic adjunction 
\begin{equation}\label{monadic}
\begin{tikzcd}
{\bU} & {\bV}
\arrow[""{name=0, anchor=center, inner sep=0}, "U"', from=1-1, to=1-2]
\arrow[""{name=1, anchor=center, inner sep=0}, "F"', curve={height=14pt}, from=1-2, to=1-1]
\arrow["\dashv"{anchor=center, rotate=-90}, draw=none, from=1, to=0]
\end{tikzcd}
\end{equation}
with $\bV$ semi-abelian. This adjunction is associated with the unique morphism $\0 \to 1$ (up to an equivalence) if and only if the unit of the adjunction is cartesian. One may always choose $\bV=(\bU \downarrow \0)$ with $U$ and $F$ defined in the obvious way, but this may not be the most convenient choice. For instance, if $\bU$ is already semi-abelian, it might be most convenient to take $\bV=\bU$. We further observe that, since $F$ is a left adjoint, $F(0)$ is an initial object of $\bU$.
\end{remark}

\begin{example}\label{monadic_ring}
Let $\bU=\Ring$ be the category of unital rings, which has the ring of integers~$\bZ$ as initial object and the zero ring $\{ 0 \}$ as terminal one. Then a monadic adjunction as in Remark~\ref{cartesian} is given by
\begin{equation}\label{eq:situation_rng}
\begin{tikzcd}
{\Ring} & {\Rng,}
\arrow[""{name=0, anchor=center, inner sep=0}, "U"', from=1-1, to=1-2]
\arrow[""{name=1, anchor=center, inner sep=0}, "F"', curve={height=16pt}, from=1-2, to=1-1]
\arrow["\dashv"{anchor=center, rotate=-90}, draw=none, from=1, to=0]
\end{tikzcd}
\end{equation}
where $\Rng$ is the semi-abelian category of not necessarily unital rings, $U$ is the forgetful functor and $F$ maps every ring $X$ to the semidirect product $\bZ \ltimes X$ with multiplication
\[
(\alpha,x)\cdot(\alpha',x')=(\alpha \alpha', xx' + \alpha x' + \alpha' x)
\] 
and unit element $(1,0_X)$. The unit $\eta \colon 1_{\Rng} \Rightarrow UF$ of the adjunction is cartesian since the map $\eta_X \colon X \to U(\bZ \ltimes X) \colon x \mapsto (0,x)$ is a kernel of
\[
UF(\tau_X) \colon U(\bZ \ltimes X) \to U(\bZ) \colon (\alpha,x) \mapsto \alpha.
\]
Thus, $F \dashv U$ is, up to an equivalence, the adjunction associated with the unique morphism $\bZ \to \{ 0 \}$ in $\Ring$.
\end{example}
 
\section{Coherent and ideal actions}\label{unitalactions}
Let $\bU$ be an ideally exact category, let $\bV$ be a semi-abelian category and let 
\begin{equation}\label{eq:situation}
\begin{tikzcd}
{\bU} & {\bV}
\arrow[""{name=0, anchor=center, inner sep=0}, "U"', from=1-1, to=1-2]
\arrow[""{name=1, anchor=center, inner sep=0}, "F"', curve={height=14pt}, from=1-2, to=1-1]
\arrow["\dashv"{anchor=center, rotate=-90}, draw=none, from=1, to=0]
\end{tikzcd}
\end{equation}
be a monadic adjunction with cartesian unit. We shall refer to such a situation as an \emph{ideally exact context}. 
Since $\bV$ is semi-abelian, internal actions can clearly be defined for objects of $\bV$.  In this section we extend this possibility to objects of $\bU$ acting via the functor~$U$.

\begin{definition}
Let $B$ be an object of $\bU$ and let $X$ be an object of $\bV$. A \emph{relative $U$-action} of $B$ on $X$ is an internal action $\xi \colon U(B)\flat X \to X$ in $\bV$.
\end{definition}

\begin{definition}
Let $\xi \colon U(B)\flat X\to X$ be a relative $U$-action. We say that $\xi$ is a \emph{coherent action} if $\xi \circ (U(\iota_B )\flat \id_{X})=\xi_{0}$\footnote{The map $U(\iota_B) \flat \id_X \colon UF(0) \flat X \to U(B) \flat X$ is the unique morphism such that
\[
\kappa_{U(B),X} \circ (U(\iota_B) \flat \id_X) = (U(\iota_B) + \id_X) \circ \kappa_{UF(0),X},
\]
where $U(\iota_B) + \id_X \colon UF(0) + X \to U(B) + X$ is induced by $U(\iota_B)$ on $UF(0)$ and by $\id_X$ on $X$.}, where $\xi_{0}\colon UF(0)\flat X \to X$ is the relative $U$-action associated with the canonical split epimorphism
\[\begin{tikzcd}
{UF(X)} & {UF(0).}
\arrow["{UF(\tau_X)}", shift left=2, from=1-1, to=1-2]
\arrow["{UF(\iota_X)}", shift left, from=1-2, to=1-1]
\end{tikzcd}\]
In other words, the following diagram in $\bV$
\[\begin{tikzcd}
{U(B)\flat X} & X \\
{UF(0)\flat X}
\arrow["\xi", from=1-1, to=1-2]
\arrow["{U(\iota_B)\flat \id_{X}}", from=2-1, to=1-1]
\arrow["{\xi_{0}}"', from=2-1, to=1-2]
\end{tikzcd}\]
is commutative.
\end{definition}
    
\begin{example}
The action $\xi_0$ is coherent since $\xi_0=\xi_0 \circ (U(\id_{F(0)})\flat \id_{X})$.
\end{example}

\begin{remark}
Let us motivate our terminology choices.
In the case of Example~\ref{monadic_ring}, the action of a unital ring $B$ on a non-unital ring $X$ is coherent when the multiplicative unit acts \emph{coherently} as in $F(X)=\mathbb{Z} \ltimes X$ (see Theorem~\ref{BAT_Rng}). More generally, in the varietal case, when $U$ is a forgetful functor into a semi-abelian variety that forgets all the constants but one, coherent actions are those for which the constants behave coherently as in $F(X)$.
\end{remark}

The next lemma, whose proof is immediate, allows us to relate coherent actions in~$\bV$ with split epimorphisms in $\bU$.

\begin{lemma}\label{unit}
Let $\xi\colon U(B)\flat X \to X$ be a relative $U$-action, and let
\[
\begin{tikzcd}
{A} & {U(B)}
\arrow["{p}", shift left, from=1-1, to=1-2]
\arrow["{s}", shift left, from=1-2, to=1-1]
\end{tikzcd}
\]
be a split epimorphism associated with $\xi$ under the equivalence \eqref{equivalence}. Then, $\xi$ is coherent if and only if there exists a morphism $f\colon UF(X)\to A$ in $\bV$ such that the following diagram
\[
\begin{tikzcd}
X & {UF(X)} & {UF(0)} \\
X & A & {U(B)}
\arrow["{\eta_{X}}", from=1-1, to=1-2]
\arrow[equal, from=1-1, to=2-1]
\arrow["{UF(\tau_X)}", shift left, from=1-2, to=1-3]
\arrow["f"', dashed, from=1-2, to=2-2]
\arrow["{UF(\iota_X)}", shift left, from=1-3, to=1-2]
\arrow["{U(\iota_B)}", from=1-3, to=2-3]
\arrow["k"', from=2-1, to=2-2]
\arrow["p", shift left, from=2-2, to=2-3]
\arrow["s", shift left, from=2-3, to=2-2]
\end{tikzcd}
\]
is a morphism of split extensions in $\bV$ (see~\cite[Lemma 2.3]{cross}). Furthermore, the square on the right is a split pullback. \noproof
\end{lemma}

We observe that, since the category $\bV$ is protomodular, by~\cite[Lemma 3.1.22]{malcev} the pair $(\eta_{X},UF(\iota_X))$ is jointly strongly epimorphic. Thus, if it exists, $f$ is uniquely determined by $\id_X$ and $U(\iota_B)$.

Recall from~\cite{rel} that a \emph{relative $U$-ideal} of an object $A'$ of $\bU$ is a morphism $k\colon X\to U(A')$ in $\bV$ such that there exists a morphism $p'\colon A'\to B$ in $\bU$ that makes the following diagram a pullback in $\bV$:
\[
\begin{tikzcd}
{X} & U(A') \\
0 & U(B).
\arrow["{k}", from=1-1, to=1-2]
\arrow["\tau_X"', from=1-1, to=2-1]
\arrow["U(p')", from=1-2, to=2-2]
\arrow["\iota_{U(B)}"', from=2-1, to=2-2]
\end{tikzcd}
\]
In other words, a $U$-ideal is a kernel in $\bV$ of a map that lives in $\bU$.

\begin{definition}\label{def_ideal}
Given an ideally exact context \eqref{eq:situation}, we say that a split epimorphism
\begin{equation}\label{ideal_split}
\begin{tikzcd}
{A} & {U(B)}
\arrow["{p}", shift left, from=1-1, to=1-2]
\arrow["{s}", shift left, from=1-2, to=1-1]
\end{tikzcd}    
\end{equation}
in $\bV$ is \emph{ideal} (relative to $U$), if there exists a split epimorphism
\begin{equation}\label{ideal_split_prime}\begin{tikzcd}
{A'} & {B}
\arrow["{p'}", shift left, from=1-1, to=1-2]
\arrow["{s'}", shift left, from=1-2, to=1-1]
\end{tikzcd}
\end{equation}
in $\bU$ and an isomorphism $\sigma\colon U(A')\to A$ such that the following diagram in $\bV$ 
\[\begin{tikzcd}
{U(A')} & {} & A \\
& {U(B)}
\arrow["\sigma", from=1-1, to=1-3]
\arrow["{U(p')}"', shift right, from=1-1, to=2-2]
\arrow["p"', from=1-3, to=2-2]
\arrow["{U(s')}"'{pos=0.3}, shift right, from=2-2, to=1-1]
\arrow["s"', shift right=2, from=2-2, to=1-3]
\end{tikzcd}\]
is commutative.
\end{definition}

\begin{remark}
A motivation for our choices of terminology: the split epimorphism~\eqref{ideal_split} is called \emph{ideal} because the object $U(B)$ acts on the domain of the relative $U$-ideal $k \colon X \to U(A')$, where $(X,k)$ is a kernel of $U(p')$.
\end{remark}

We recall that a functor $U\colon \bU \to \bV$ is \emph{full on isomorphisms} if, given an isomorphism $\alpha\colon U(A) \to U(B)$ in $\bV$, there exists an isomorphism $\beta\colon A \to B$ in $\bU$ such that $U(\beta)=\alpha$. 

\begin{lemma}\label{lm:unique_xi}
Given an ideally exact context \eqref{eq:situation}, suppose that the functor $U$ is faithful and full on isomorphisms. If the split epimorphism \eqref{ideal_split} of Definition~\ref{def_ideal} is ideal, then the split epimorphism \eqref{ideal_split_prime} is essentially unique.
\end{lemma}
\begin{proof}
Suppose there exist two split epimorphisms
\[\begin{tikzcd}
{A'} & {B}
\arrow["{p'}", shift left, from=1-1, to=1-2]
\arrow["{s'}", shift left, from=1-2, to=1-1]
\end{tikzcd}\quad \text{ and } \quad
\begin{tikzcd}
{A''} & {B}
\arrow["{p''}", shift left, from=1-1, to=1-2]
\arrow["{s''}", shift left, from=1-2, to=1-1]
\end{tikzcd}\]
in $\bU$, and two isomorphisms $\sigma_1 \colon U(A') \to A$, $\sigma_2 \colon U(A'') \to A$ such that the following diagrams in $\bV$
\[\begin{tikzcd}
{U(A')} & {} & A \\
& {U(B)}
\arrow["\sigma_1", from=1-1, to=1-3]
\arrow["{U(p')}"', shift right, from=1-1, to=2-2]
\arrow["p"', from=1-3, to=2-2]
\arrow["{U(s')}"'{pos=0.3}, shift right, from=2-2, to=1-1]
\arrow["s"', shift right=2, from=2-2, to=1-3]
\end{tikzcd}
\qquad
\begin{tikzcd}
{U(A'')} & {} & A \\
& {U(B)}
\arrow["\sigma_2", from=1-1, to=1-3]
\arrow["{U(p'')}"', shift right, from=1-1, to=2-2]
\arrow["p"', from=1-3, to=2-2]
\arrow["{U(s'')}"'{pos=0.3}, shift right, from=2-2, to=1-1]
\arrow["s"', shift right=2, from=2-2, to=1-3]
\end{tikzcd}
\]
commute. Thus, if $\sigma=\sigma_2^{-1} \circ \sigma_1$, the diagram in $\bV$
\[
\begin{tikzcd}
{U(A')} & {} & {U(A'')}\\
& {U(B)}
\arrow["\sigma", from=1-1, to=1-3]
\arrow["{U(p')}"', shift right, from=1-1, to=2-2]
\arrow["U(p'')"', from=1-3, to=2-2]
\arrow["{U(s')}"'{pos=0.3}, shift right, from=2-2, to=1-1]
\arrow["U(s'')"', shift right=2, from=2-2, to=1-3]
\end{tikzcd}
\]
is commutative. Since $U$ is faithful and full on isomorphisms, there exists a morphism $\sigma' \colon A' \to A''$ in $\bU$, such that $U(\sigma')=\sigma$. By faithfulness of the functor $U$, $\sigma'$ is a morphism in the category $\Pt_\bU(B)$, i.e., the split epimorphisms
\[\begin{tikzcd}
{A'} & {B}
\arrow["{p'}", shift left, from=1-1, to=1-2]
\arrow["{s'}", shift left, from=1-2, to=1-1]
\end{tikzcd}\quad \text{ and } \quad
\begin{tikzcd}
{A''} & {B}
\arrow["{p''}", shift left, from=1-1, to=1-2]
\arrow["{s''}", shift left, from=1-2, to=1-1]
\end{tikzcd}\]
are isomorphic.
\end{proof}

\begin{definition}
Given an ideally exact context \eqref{eq:situation}, let $B$ be an object of $\bU$. A morphism 
\[
\begin{tikzcd}
{A_1} & {} & {A_2} \\
& {U(B)}
\arrow["h", from=1-1, to=1-3]
\arrow["{p_1}"', shift right, from=1-1, to=2-2]
\arrow["p_2"', from=1-3, to=2-2]
\arrow["{s_1}"'{pos=0.3}, shift right, from=2-2, to=1-1]
\arrow["s_2"', shift right=2, from=2-2, to=1-3]
\end{tikzcd}
\]
between ideal split epimorphisms over $U(B)$ is an \emph{ideal morphism} if there exists a morphism
\[
\begin{tikzcd}
{A_1'} & {} & {A_2'} \\
& {B}
\arrow["h'", from=1-1, to=1-3]
\arrow["{p_1'}"', shift right, from=1-1, to=2-2]
\arrow["p_2'"', from=1-3, to=2-2]
\arrow["{s_1'}"'{pos=0.3}, shift right, from=2-2, to=1-1]
\arrow["s_2'"', shift right=2, from=2-2, to=1-3]
\end{tikzcd}
\]
between the corresponding split epimorphisms over $B$ of Definition~\ref{def_ideal}, such that $U(h')=h$.
\end{definition}

\begin{definition}
Let $\xi \colon U(B)\flat X\to X$ be a relative $U$-action. We say that $\xi$ is an \emph{ideal action} if any corresponding split epimorphism
\[
\begin{tikzcd}
{A} & {U(B)}
\arrow["{p}", shift left, from=1-1, to=1-2]
\arrow["{s}", shift left, from=1-2, to=1-1]
\end{tikzcd}
\]
under the equivalence \eqref{equivalence} is ideal.
\end{definition}
Similarly, an ideal morphism of ideal actions comes from an ideal morphism of the corresponding ideal split epimorphisms.

\smallskip
The following theorem shows a connection between coherent actions and ideal actions in any ideally exact context.
	
\begin{theorem}\label{thmunital} 
Consider an ideally exact context \eqref{eq:situation} and let $\xi\colon U(B)\flat X\to X$ be a relative $U$-action. If $\xi$ is ideal, then $\xi$ is coherent.
\end{theorem}

\begin{proof}
Let \eqref{ideal_split} be a split epimorphism associated with $\xi$ under the equivalence~\eqref{equivalence}. Since $\xi$ is ideal, then there exists in $\bU$ a split epimorphism
\[\begin{tikzcd}
{A'} & {B}
\arrow["{p'}", shift left, from=1-1, to=1-2]
\arrow["{s'}", shift left, from=1-2, to=1-1]
\end{tikzcd}\] 
and an isomorphism $\sigma \colon U(A') \to A$ such that $U(p')=p\circ \sigma $ and $\sigma \circ U(s')=s$.

Let $(X,k)$ be a kernel of $p$, and let $k'=\sigma^{-1}\circ k$, so that $(X,k')$ is a kernel of $U(p')$. Since the unit $\eta$ of the adjunction is cartesian, then $\eta_{X}\colon X \to UF(X)$ defines a kernel of $UF(\tau_X)\colon UF(X)\to UF(0)$. Moreover, if $f\colon F(X) \to A'$ is the morphism in $\bU$ given by the universal property of the unit $\eta$ with respect to the kernel $k'\colon X \to U(A')$ of $U(p')$, that is
\[
\begin{tikzcd}
X & {UF(X)} & {F(X)} \\
{U(A')} && {A'}
\arrow["{\eta_{X}}", from=1-1, to=1-2]
\arrow["k'"', from=1-1, to=2-1]
\arrow["U(f)", dashed, from=1-2, to=2-1]
\arrow["f", dashed, from=1-3, to=2-3]
\end{tikzcd}
\]
then the diagram
\[
\begin{tikzcd}
X & {UF(X)} & {UF(0)} \\
X & {U(A')} & {U(B)}
\arrow["{\eta_{X}}", from=1-1, to=1-2]
\arrow[equal, from=1-1, to=2-1]
\arrow["{UF(\tau_X)}", shift left, from=1-2, to=1-3]
\arrow["U(f)"', dashed, from=1-2, to=2-2]
\arrow["{UF(\iota_X)}", shift left, from=1-3, to=1-2]
\arrow["{U(\iota_B)}", from=1-3, to=2-3]
\arrow["k'"', from=2-1, to=2-2]
\arrow["{U(p')}", shift left, from=2-2, to=2-3]
\arrow["{U(s')}", shift left, from=2-3, to=2-2]
\end{tikzcd}
\]
is commutative. Indeed, one has $U(f)\circ \eta_{X}=k'$. Moreover, by protomodularity, the pair $(\eta_{X}, UF(\iota_X))$ is jointly strongly epimorphic. Thus, we may prove that the morphisms $U(\iota_B)\circ UF(\tau_X)$ and $U(p')\circ U(f)$ are equal by composing them with the pair $(\eta_{X}, UF(\iota_X))$. Indeed, one has
\[U(\iota_B)\circ UF(\tau_X)\circ \eta_{X}=0=U(p')\circ k'=U(p')\circ U(f)\circ \eta_{X}.\]
Moreover, since $F(0)$ is initial in $\bU$, one has
\[
s' \circ \iota_B=f \circ F(\iota_X),
\]
and therefore
\[
U(\iota_B)\circ UF(\tau_X)\circ UF(\iota_X)=U(\iota_B)=U(p')\circ U(s')\circ U(\iota_B)= U(p')\circ U(f)\circ UF(\iota_X).
\]
Hence $U(\iota_B)\circ UF(\tau_X)=U(p')\circ U(f)$. As a consequence, by composing the diagrams below
\[\begin{tikzcd}
X & {UF(X)} & {UF(0)} \\
X & {U(A')} & {U(B)} \\
X & A & {U(B)}
\arrow["{\eta_{X}}", from=1-1, to=1-2]
\arrow[equal, from=1-1, to=2-1]
\arrow["{UF(\tau_X)}", shift left, from=1-2, to=1-3]
\arrow["{U(f)}"', from=1-2, to=2-2]
\arrow["{UF(\iota_X)}", shift left, from=1-3, to=1-2]
\arrow["{U(\iota_B)}", from=1-3, to=2-3]
\arrow["{k'}"', from=2-1, to=2-2]
\arrow[equal, from=2-1, to=3-1]
\arrow["{U(p')}", shift left, from=2-2, to=2-3]
\arrow["\sigma"', from=2-2, to=3-2]
\arrow["{U(s')}", shift left, from=2-3, to=2-2]
\arrow[equal, from=2-3, to=3-3]
\arrow["k"', from=3-1, to=3-2]
\arrow["p", shift left, from=3-2, to=3-3]
\arrow["s", shift left, from=3-3, to=3-2]
\end{tikzcd}\]
we have that the action $\xi$ is coherent.
\end{proof}

At present, the authors do not know whether the converse of Theorem~\ref{thmunital} holds in any ideally exact context. This is currently under investigation. However, when it holds, it establishes a convenient setting to study a well-behaved notion of action. This motivates the following definition.
\begin{definition}
An ideally exact context \eqref{eq:situation} admits a \emph{good theory of actions} (or, it is BAT\footnote{The acronym BAT is inspired by the notion of BIT-variety, where BIT stands for {\bf B}uona (good, in Italian) {\bf I}deal {\bf T}heory, introduced by A.~Ursini in~\cite{BIT}. Analogously, BAT stands for {\bf B}uona {\bf A}ction {\bf T}heory.}, for short) if all coherent actions are ideal, and all morphisms of such actions are ideal.
\end{definition}

In \Cref{cases} we present some case studies which have a good theory of actions.

\section{A connection with semidirect products in ideally exact categories}\label{semidirect}

During the preparation of this manuscript, a paper by G.~Janelidze appeared~\cite{semidir_IdeallyExact}. In his work, the author extended the notion of semidirect product from the semi-abelian to the ideally exact context. In this section, we make explicit some connections between G.~Janelidze's approach to the subject and our notions of relative $U$-action.

As we recalled in the introduction, the notion of internal action expresses its full potential when the base category is pointed, since in this case, split epimorphisms can be recovered by means of actions on their kernels. Actually, if the base category is not pointed, one can only express the split epimorphisms over an object $B$ in terms of algebras on a split epimorphism over another object $E$, whenever there is a morphism $E \to B$. This is not a major drawback, unless the latter are less involved than the former, as for instance when the category $\bU$ is pointed and $E$ is the initial$=$terminal object.

However, even if $\bU$ is not pointed, not all is lost. If we suppose that $\bU$ has an initial object $\0$ and pushouts of split monomorphisms along any map exist (and this is the case, since $\bU$ is ideally exact), protomodularity can be stated just in terms of the initial arrows alone: all the $f^*$'s are conservative if and only if just the $\iota^*$'s are. This suggests considering initial arrows as \emph{canonical} for describing semidirect products, so that an action becomes an algebra on $\Pt_{\bU}(\0)=(\bU\downarrow \0)$.

Before tackling actions directly, let us remain on the categories of points. Let $F\dashv U$ be a monadic adjunction as in \eqref{eq:situation}, with semi-abelian codomain and cartesian unit, and consider the diagram
\begin{equation}\label{eq:main}
\begin{aligned}
\xymatrix@C=10ex@R=10ex{
\Pt_{\bU}(B)\ar[r]_-{U'}^-{\bot} \ar[d]_-{\iota_B^*}^{\vdash}
&\Pt_{\bV}(U(B))\ar[d]_{U(\iota_B)^*}^{\vdash}\ar@/_2.7ex/[l]_-{F'}
\\
\Pt_{\bU}(F(0))\ar[r]_-{U_0'}^-{\bot}\ar@/_2.7ex/[u]_{B+(-)}
&\Pt_{\bV}(UF(0))\ar@/_2.7ex/[u]_{U(B)+_{UF(0)}(-)}\ar@/_2.7ex/[l]_-{F_0'}
}
\end{aligned}
\end{equation}
where
\begin{itemize}
\item[(i)] $U'$ maps any split epimorphism 
\[
\begin{tikzcd}
{A'} & {B}
\arrow["{p'}", shift left, from=1-1, to=1-2]
\arrow["{s'}", shift left, from=1-2, to=1-1]
\end{tikzcd}
\]
in $\bU$, to the split epimorphism 
\[
\begin{tikzcd}
{U(A')} & {U(B)}
\arrow["{U(p')}", shift left, from=1-1, to=1-2]
\arrow["{U(s')}", shift left, from=1-2, to=1-1]
\end{tikzcd}
\]
in $\bV$.
\item[(ii)] For any split epimorphism 
\[
\begin{tikzcd}
{A} & {U(B)}
\arrow["{p}", shift left, from=1-1, to=1-2]
\arrow["{s}", shift left, from=1-2, to=1-1]
\end{tikzcd}
\]
in $\bV$, $F'$ is defined by taking the pushout of $F(s)$ along the counit component $\varepsilon_B$ of the adjunction:
\[
\xymatrix@C=10ex{
F(A)\ar@<-.5ex>[d]_{F(p)}\ar[r]^-{i_2}
&B+_{FU(B)}F(A)\ar@<-.5ex>[d]_{[\id_B,\varepsilon_B \circ F(p)]}
\\
FU(B)\ar@<-.5ex>[u]_{F(s)}\ar[r]_-{\varepsilon_B}
&B. \ar@<-.5ex>[u]_{i_1}
}
\]
\item[(iii)] The definitions of $U'_0$ and $F'_0$ are similar, with $F(0)$ instead of $B$.
\end{itemize}
Straightforward calculations show that in diagram \eqref{eq:main}, the square of right adjoints commutes (up to isomorphism), and the same can be said for the square made of left adjoints. Notice that $F'\dashv U'$ are the adjunctions defined in~\cite[Section 2]{semidir_IdeallyExact}, whose notation we are adopting.

Now, consider the diagram
\begin{equation}\label{eq:main2}
\begin{aligned}
\xymatrix@C=14ex@R=10ex{
\Pt_{\bU}(F(0))\ar[r]_-{U_0'}^-{\bot}\ar[d]_-{U^{F(0)}}^-{\vdash}
&\Pt_{\bV}(UF(0))\ar@/_2.7ex/[l]_-{F_0'}\ar[dl]_{U_0''}^-{\vdash}
\\
\bV\ar@/_2.7ex/[u]_-{F^{F(0)}}\ar@/_2.7ex/[ur]_-{F_0''}
}
\end{aligned}
\end{equation}
where 
\[
U''_0\left(\raisebox{\depth/2}{\small \xymatrix{
A\ar@<+.5ex>[d]^{p}
\\
UF(0)\ar@<+.5ex>[u]^{s}
}}\right)=\ker p \,,\qquad F_0''(X)=\raisebox{\depth/2}{\small \xymatrix{
UF(0)+X\ar@<+.5ex>[d]^{[\id_{UF(0)},0]}
\\
UF(0)\ar@<+.5ex>[u]^{i_1}
}}\,,
\]
\[
U^{F(0)}=U_0''\circ U_0'\,,\qquad
F^{F(0)}=F_0'\circ F_0''\,.
\]
We notice that:
\begin{itemize}
\item[-] the adjunctions $F_0'\dashv U_0'$ and $F_0''\dashv U_0''$ are nothing but the ones defined in~\cite[Section 2]{semidir_IdeallyExact}, with $\mathcal A=\bU$, $\mathcal X=\bV$ and $F(0)$ replacing $B$;
\item[-] $F^{F(0)}\dashv U^{F(0)} $ is an adjoint equivalence (see~\cite[Theorem 3.17]{rel}, see also~\cite[Theorem 2.6]{IdeallyExact}).
\end{itemize}
Now, we can paste diagram \eqref{eq:main} with diagram \eqref{eq:main2}, and define
\[
U''=U_0''\circ U(\iota_B)^*\,,\qquad F''= (U(B)+_{UF(0)}(-))\circ F_0''\,.
\]
\[
U^{B}=U''\circ U'\,,\qquad
F^{B}=F'\circ F''\,.
\]
We notice that:
\begin{itemize}
\item[-] the adjunctions $F'\dashv U'$, $F''\dashv U''$ and $F^B\dashv U^B$ are nothing but the ones defined in~\cite[Section 2]{semidir_IdeallyExact}, with $\mathcal A=\bU$, $\mathcal X=\bV$; 
\item[-] there are natural isomorphisms $U^B\cong U^{F(0)}\circ \iota_B^*$ and $F^B\cong (B+(-))\circ F^{F(0)}$. 
\end{itemize}
In fact, $U^B$ is monadic (see~\cite[Theorem 2.1]{semidir_IdeallyExact}), so that one can describe points over $B$ as algebras on objects of $\bV$. More precisely, we let $B\#(-) = U^B\circ F^B$, i.e., for an object $X$ of $\bV$, $B\#X$ is given by a kernel 
\[
\Tilde{\kappa}_{B,X} \colon B\#X \to U(B+F(X))
\]
of the morphism
\[
U([\id_B,\iota_B \circ F(\tau_X)])\colon U(B+F(X)) \to U(B).
\]
Notice that $(-)\#(-)$ is functorial on both components, just like $(-)\flat(-)$ is. Furthermore, one can define a natural transformation $\gamma=\gamma_{B,X}$ as the unique dashed arrow making the following diagram commute: 
\begin{equation}\label{eq:gamma}
\begin{aligned}
\xymatrix@C=14ex{
U(B)\flat X\ar[r]^-{\kappa_{U(B),X}}\ar@{-->}[d]_{\gamma_{B,X}}
&U(B)+X\ar[r]^-{[\id_{U(B)},0]}\ar[d]_{[U(i_1),U(i_2)\circ\eta_X]}
&U(B)\ar@{=}[d]
\\
B\#X\ar[r]_-{\Tilde{\kappa}_{B,X}}
&U(B+F(X))\ar[r]_-{U([\id_B,\iota_B \circ F(\tau_X)])}
&U(B).
}
\end{aligned}  
\end{equation}

Indeed, $\gamma_{B}=\gamma_{B,-}\colon B\#(-)\Rightarrow B\flat(-)$ is a morphism of monads, and the functor $U'\colon \Pt_{\bU}(B) \to \Pt_{\bV}(U(B))$ translates into the functor $\overline U=(-)\circ \gamma_{B}\colon \bV^{B\#} \to \bV^{U(B)\flat }$ between the categories of algebras. Therefore, given an algebra $\xi'\colon B\#X \to X$, $\overline U(\xi')=\xi'\circ \gamma_{B,X} \colon U(B) \flat X \to X$ is an internal action in $\bV$.

\begin{lemma}\label{lm:azioni_in_algebre}
Given an ideally exact context \eqref{eq:situation}, a relative $U$-action $\xi\colon U(B)\flat X \to X$ is ideal if and only if there exists a $B\#$-algebra $\xi'\colon B\#X \to X$ such that $\xi'\circ \gamma_{B,X}=\xi$. 
\[
\xymatrix@C=16ex{
B\#X\ar@{-->}[dr]^-{\xi'}
\\
U(B)\flat X\ar[u]^{\gamma_{B,X}}\ar[r]_-{\xi}
&X.
}
\] \noproof
\end{lemma}

Our purpose is now to translate in terms of algebras the arguments given above concerning the categories of points. Let us begin from the square of right adjoints in~\eqref{eq:main}. It plainly translates into the following commutative square (whereas the former commutes only up to natural isomorphisms):
\begin{equation}\label{eq:algebras}
\begin{aligned}
\xymatrix{
\bV^{B\#}\ar[r]^-{\overline U}\ar[d]_{(-)\circ (\iota_B \# \id)}
&\bV^{U(B)\flat}\ar[d]^{(-)\circ (U(\iota_B)\flat \id)}
\\
\bV^{F(0)\#}\ar[r]_-{\overline U_0}
&\bV^{UF(0)\flat}
}
\end{aligned}    
\end{equation}
where, for any object $X$ of $\bV$, $\iota_B \# \id_X \colon F(0) \# X \to B \# X$ is the unique morphism such that
\[
\Tilde{\kappa}_{B,X} \circ (\iota_B \# \id_X) = (U(\iota_B) + \id_{F(X)}) \circ \Tilde{\kappa}_{F(0),X}.
\]

The $F(0)$ case deserves some analysis: indeed, since the left adjoints preserve colimits, one has 
\[
F(0)+F(X)\cong F(0+X)\cong F(X),
\]
so that for $B=F(0)$, diagram \eqref{eq:gamma} turns into
\[\xymatrix@C=14ex{
UF(0)\flat X\ar[r]^-{\kappa_{UF(0),X}}\ar@{-->}[d]_{\gamma_{F(0),X}}^{=\xi_0}
&UF(0)+X\ar[r]^-{[\id_{UF(0)},0]}\ar[d]_{[UF(\iota_X),\eta_X]}
&UF(0)\ar@{=}[d]
\\
F(0)\#X=X\ar[r]_-{\eta_X}
&UF(X)\ar[r]_-{UF(\tau_X)}
&UF(0).
}
\]
Therefore, it becomes evident that $\gamma_{F(0),X}$ coincides with the canonical action $\xi_{0}$ of $UF(0)$ on $X$.
As a consequence of the previous discussion, diagram \eqref{eq:algebras} simplifies:
\begin{equation}\label{eq:algebras2}
\begin{aligned}
\xymatrix@C=10ex{
\bV^{B\#}\ar[r]^-{(-)\circ\gamma_B}\ar[d]_{S}
&\bV^{U(B)\flat}\ar[d]^{(-)\circ (U(\iota_B)\flat \id)}
\\
\bV\ar[r]_-{[\xi_0]}
&\bV^{UF(0)\flat}
}
\end{aligned}    
\end{equation}
where $S$ is the canonical forgetful functor, and $[\xi_0](X)=\xi_{0}$.
    
\begin{proposition}\label{prop_char}
Given an ideally exact context \eqref{eq:situation} such that $U$ is full on isomorphisms, let us fix an object $B$ of $\bU$. Then, the following statements are equivalent:
\begin{itemize}
\item [$(i)$] All the coherent relative $U$-actions $U(B) \flat X \to X$ are ideal and all the morphisms of such actions are ideal;
\item [$(ii)$] The diagram \eqref{eq:algebras2} is a pullback;
\item [$(iii)$] The square of right adjoints in \eqref{eq:main} is a pseudopullback.
\end{itemize}
\end{proposition}
\begin{proof}
$(ii)\Rightarrow(i)$\ \ Let us consider the comparison functor $H=\langle S, (-)\circ\gamma_B \rangle$ below
\begin{equation}\label{eq:pb_algebras}
\begin{aligned}
\xymatrix@C=10ex{
\bV^{B\#}\ar@/^3ex/[drr]^-{(-)\circ\gamma_B}\ar@/_5ex/[ddr]_{S}\ar@{-->}[dr]^-{H}
\\
&\bV\times_{\bV^{UF(0)\flat}}\bV^{U(B)\flat}
\ar[r]^-{p_2}\ar[d]_{p_1}
&\bV^{U(B)\flat}\ar[d]^{(-)\circ (U(\iota_B)\flat \id)}
\\
&\bV\ar[r]_-{[\xi_0]}
&\bV^{UF(0)\flat}.
}
\end{aligned}    
\end{equation}
Since diagram \eqref{eq:algebras2} is a pullback, then $H$ is an isomorphism. Hence, we can define $\xi'=H^{-1}(X,\xi)$, for any coherent relative $U$-action $\xi \colon U(B) \flat X \to X$. Thus, by Lemma~\ref{lm:azioni_in_algebre}, the action $\xi$ is ideal. One may check that the same argument applies also to morphisms.

\smallskip
$(i)\Rightarrow(ii)$\ \ Conversely, assume that all coherent actions are ideals and that all the morphisms between coherent actions are ideal, and consider a coherent relative $U$-action $\xi \colon U(B) \flat X \to X$. Then, by the algebras version of Lemma~\ref{lm:unique_xi}, there is a unique $B\#$-algebra $\xi'$ such that $\xi'\circ \gamma_{B,X}=\xi$. Thus, the functor $H$ is bijective on objects. Since morphisms of $B\#$-algebras and those of $U(B)\flat$-algebras are determined by the same underlying maps, the proof that $H$ is bijective on arrows is trivial.

\smallskip
$(ii)\Leftrightarrow(iii)$\ \ Since the functor $[\xi_0]$ has invertible-path lifting, by \cite[Theorem 1]{JoyalStreet}, diagram \eqref{eq:algebras2} is a pullback if and only if it is a pseudopullback. On the other hand, by identifying $\bV^{F(0)\#}$ with $\bV$, one sees that \eqref{eq:algebras2} is a pseudopullback if and only if \eqref{eq:algebras} is, and this happens precisely when the equivalent diagram \eqref{eq:main} is a pseudopullback. 
\end{proof}

The following statement provides a characterisation of a class of ideally exact contexts admitting a good theory of actions.

\begin{corollary}\label{cor:main}
Consider an ideally exact context \eqref{eq:situation} such that the functor $U$ is full on isomorphisms. Then the following statements are equivalent:
\begin{itemize}
\item [$(i)$] the ideally exact context \eqref{eq:situation} is BAT;
\item [$(ii)$] for any object $B$ of $\bU$, \eqref{eq:algebras2} is a pullback;
\item [$(iii)$] for any object $B$ of $\bU$, the square of right adjoints in \eqref{eq:main} is a pseudopullback.
\end{itemize}   
\end{corollary}
\begin{proof}
The proof is an immediate consequence of Proposition~\ref{prop_char}
\end{proof}
\begin{remark}\label{rm:this_remark}
Given an ideally exact context \eqref{eq:situation}, consider a morphism $f\colon E\to B$ in $\bU$ together with the following commutative diagram:
\begin{equation}\label{eq:U_cartesian}
\begin{aligned}
\xymatrix{
\bV^{B\#}\ar[r]^-{\overline{U}_B}\ar[d]_{(-)\circ (f\# \id)}\ar@<-12ex>[dd]_{(-)\circ (\iota_B\# \id)}
&\bV^{U(B)\flat}\ar[d]^{(-)\circ (U(f)\flat \id)}\ar@<+14ex>[dd]^{(-)\circ (U(\iota_B)\flat \id)}
\\
\bV^{E\#}\ar[r]^-{\overline{U}_E}\ar[d]_{(-)\circ (\iota_E\# \id)}
&\bV^{U(E)\flat}\ar[d]^{(-)\circ (U(\iota_E)\flat \id)}
\\
\bV^{F(0)\#}\ar[r]_-{\overline{U}_{F(0)}}
&\bV^{UF(0)\flat}.
}
\end{aligned}    
\end{equation}
By the usual cancellation property, if the outer and the bottom squares are pullbacks, so is the top square. This fact has interesting consequences. 

Indeed, the upper square in \eqref{eq:U_cartesian} can be interpreted as a naturality square of a natural transformation between two contravariant functors, namely
\begin{align*}
\#\text{-alg}\colon& \bU^{\op}\to \Cat \qquad B\mapsto \bV^{B\#},\\
\flat\text{-alg}\colon& \bU^{\op}\to \Cat \qquad B\mapsto \bV^{U(B)\flat}.
\end{align*}
The $\overline{U}$'s in the diagram are the components of the natural transformation \[\overline U\colon \#\text{-alg} \Rightarrow \flat\text{-alg}\] induced by the monad morphism $\gamma=\gamma_{B,X}$. Moreover, one can define a component-wise left adjoint by lifting the left adjoints $\overline F_0 \dashv \overline U_0$.

Natural transformations whose naturality squares are pullbacks are called \emph{carte\-sian} in the literature---an unfortunate terminology choice, in our setting. However, cartesian natural transformations have an important feature: if the domain category has a terminal object, then all the components of the transformation are determined by the component on the terminal object. This is exactly what happens here, since a terminal object of $\bU^{\operatorname{op}}$ is initial in $\bU$. 
For this reason, let us call \emph{$0$-determined} an ideal context where $\overline U$ yields a cartesian natural transformation. We can add the following statement to the equivalent conditions of Corollary~\ref{cor:main}.

\begin{itemize}
\item [$(iv)$] \emph{The ideally exact context \eqref{eq:situation} is $0$-determined.}
\end{itemize}
All we have just said about actions and algebras can be stated as well for the categories of points, but, in such a case, the functors
\[
U'\colon \Pt_{\bU}(B)\to \Pt_{\bV}(U(B))
\]
would be components of a pseudocartesian\footnote{This notion, namely a pseudonatural transformation such that all the naturality squares are pseudopullbacks, does not seem to have received much attention in the literature.} transformation.
\end{remark}

We leave for future work the investigation of the consequences of Remark~\ref{rm:this_remark}.

\section{Case studies}\label{cases}
Our aim is now to characterise coherent actions in the cases where $\bU$ is a unit-closed variety of non-associative algebras, where $\bU$ is the variety of MV-algebras or the variety of product algebras, and where $\bU$ is the dual of the topos of pointed sets. Our investigation shows that these ideally exact contexts have a good theory of actions.

\subsection{Varieties of non-associative algebras}\label{Var}

The aim of this section is to describe coherent actions in the framework of \emph{varieties of non-associative algebras} over a field $\bF$. We think of those as collections of algebras satisfying a chosen set of identities. We refer the reader to~\cite{VdL-NAA} for more details.

A \emph{non-associative algebra} over $\bF$ is a vector space $X$ equipped with a bilinear operation
\[
X \times X \to X \colon (x,y) \mapsto xy,
\]
called the \emph{multiplication}. The category of all non-associative algebras over $\bF$ is denoted by~$\Alg$ and its morphisms are the linear maps that preserve the multiplication.

\begin{definition}\label{def identity variety}
An \emph{identity} of a non-associative algebra $X$ is a non-associative polynomial $\varphi=\varphi(x_1,\dots ,x_n)$ such that $\varphi(x_1, \dots , x_n)=0$ for all $x_1, \ldots, x_n \in X$. We say that the algebra $X$ \emph{satisfies} the identity $\varphi$.
\end{definition}

\begin{definition}
Let $I$ be a set of identities. The \emph{variety of non-associative algebras} $\bV$ determined by $I$ is the class of all non-associative algebras that satisfy all the identities of $I$.
\end{definition}

We observe that every variety of non-associative algebras $\bV$ forms a full subcategory of~$\Alg$ and is a semi-abelian category.

\begin{examples}\label{Examples varieties}{\ }
\begin{enumerate}
\item $\AbAlg$ is the variety of \emph{abelian} algebras, which is determined by the identity $xy=0$. It is isomorphic to the category $\Vect$ of $\bF$-vector spaces and it is the only non-trivial variety of algebras which is an abelian category.

\item $\Assoc$ is the variety of \emph{associative} algebras, which is determined by \emph{associativity} $x(yz)=(xy)z$.

\item $\CAssoc$ is the subvariety of $\Assoc$ of commutative associative algebras.

\item $\Lie$ is the variety of \emph{Lie algebras}, which is determined by $x^2=0$ and the \emph{Jacobi identity}, that is $x(yz)+y(zx)+z(xy)=0$.

\item $\Leib$ is the variety of \emph{(right) Leibniz algebras}~\cite{loday}, which is determined by the \emph{(right) Leibniz identity}, that is $(xy)z-(xz)y-x(yz)=0$.

\item $\Alt$ is the variety of \emph{alternative algebras}, which is determined by the identities $(yx)x-yx^2=0$ and $x(xy)-x^2y=0$. We recall that every associative algebra is alternative, while an example of an alternative algebra which is not associative is given by the \emph{octonions} $\mathbb{O}$.
\end{enumerate}
\end{examples}

In order to study the notion of coherent actions for a variety of non-associative algebras~$\bV$, we need to work with the so-called \emph{unit-closed} varieties of algebras.

\begin{definition}\cite{unital}
A variety of non-associative algebras $\bV$ is said to be \emph{unit-closed} if, for any algebra $X$ of $\bV$, the algebra $\langle X,1 \rangle$ obtained by adjoining to $X$ the external element $1$, together with the identities $x \cdot 1 = 1 \cdot x = x$, is still an object of~$\bV$. 
\end{definition}

For instance, the varieties $\Assoc$, $\CAssoc$ and $\Alt$ are unit-closed, while the varieties $\Leib$, or any variety of anti-commutative algebras over a field of characteristic different from $2$, such as the category $\Lie$ of Lie algebras, are examples of not unit-closed varieties. Thus, the condition of being unit-closed is related to the set of identities which determine the variety $\bV$.

When a variety $\bV$ is unit-closed, it is possible to define the subcategory $\bV_1$ of unital algebras of $\bV$, with the arrows being the algebra morphisms of $\bV$ that preserve the unit element. Of course, $\bV_1$ is an ideally exact category and it is not pointed, since the initial object is the field $\bF$, while the terminal one is the zero algebra $\{ 0 \}$.

\begin{remark}
Let $\bV$ be a unit-closed variety of non-associative algebras. In a similar way to what happens in the category of rings (see Example~\ref{monadic_ring}), a monadic adjunction associated with the unique morphism $\bF \to \{ 0 \}$ in $\bV_1$ is
\begin{equation}\label{eq:situation_alg}
\begin{tikzcd}
{\bV_1} & {\bV,}
\arrow[""{name=0, anchor=center, inner sep=0}, "U"', from=1-1, to=1-2]
\arrow[""{name=1, anchor=center, inner sep=0}, "F"', curve={height=14pt}, from=1-2, to=1-1]
\arrow["\dashv"{anchor=center, rotate=-90}, draw=none, from=1, to=0]
\end{tikzcd}
\end{equation}
where $U$ is the forgetful functor and $F$ maps every algebra $X$ of $\bV$ to the semidirect product $\bF \ltimes X$ with multiplication
\[
(\alpha,x)\cdot(\alpha',x')=(\alpha \alpha', xx' + \alpha x' + \alpha' x)
\] 
and unit element $(1,0_X)$.
\end{remark}

We can now provide a characterisation of \emph{coherent actions} in the context of unit-closed varieties of non-associative algebras. We refer the reader to~\cite{WRAAlg, CigoliManciniMetere, XabiMancini, Tim} for a complete description of actions and their representability in varieties of non-associative algebras.

At first, we observe that the converse of Theorem~\ref{thmunital} holds for any unit-closed variety~$\bV$.

\begin{proposition}\label{conv_algebras}
Let $\bV$ be a unit-closed variety of non-associative algebras and consider the ideally exact context \eqref{eq:situation_alg}. Let
\[
\xi\colon U(B)\flat X\to X
\]
be a relative $U$-action in $\bV$. If $\xi$ is coherent, then $\xi$ is an ideal action.
\end{proposition}

\begin{proof}
Let $\xi \colon U(B) \flat X \to X$ be a coherent action and let 
\begin{equation*}
\begin{tikzcd}
X \arrow [r, "k"]
& A \arrow[r, shift left, "p"] &
U(B) \ar[l, shift left, "s"]
\end{tikzcd}
\end{equation*}
be a split extension associated with $\xi$. Since $\xi$ is coherent, there exists a morphism
\[
f \colon U(\bF \ltimes X) \to A
\]
such that the following diagram in $\bV$ 
\begin{equation*}
\begin{tikzcd}
X \arrow [r, "\eta_{X}"] \ar[d, equal]
&U(\bF \ltimes X) \arrow[r, shift left, "UF(\tau_X)"] \ar[d, dashed, "f"']&
U(\bF) \ar[l, shift left, "UF(\iota_X)"]\ar[d, "U(\iota_B)"]\\
X \arrow [r, "k"]
& A \arrow[r, shift left, "p"] &
U(B) \ar[l, shift left, "s"]
\end{tikzcd}
\end{equation*}
is commutative. Protomodularity implies that $f$ is the unique arrow induced by $\id_X$ and $U(\iota_B)$, i.e., $f(\alpha,x)=s(\alpha 1_B) + k(x)$ for every $(\alpha,x) \in \bF \ltimes X$. 

We observe that $A$ is a unital algebra with unit $s(1_B)$. Indeed, every $a \in A$ may be written as $s(b)+k(x)$, for some $b \in B$ and $x \in X$, and
\[
s(1_B)a=s(1_B)(s(b)+k(x))=s(1_B)s(b)+s(1_B)k(x)=s(b)+s(1_B)k(x)=a
\]
since
\[
s(1_B)k(x)=f(1,0_X)f(0,x)=f((1,0_X)\cdot(0,x))=f(0,x)=k(x).
\]
In a similar way, one may check that $as(1_B)=a$. Hence, we have proved that, if
\[
\begin{tikzcd}
{A} & {U(B)}
\arrow["{p}", shift left, from=1-1, to=1-2]
\arrow["{s}", shift left, from=1-2, to=1-1]
\end{tikzcd}
\] 
is a split epimorphism in $\bV$ associated with $\xi$ under the equivalence \eqref{equivalence}, then it is in fact a split epimorphism in $\bV_1$, since $A$ is a unital algebra and both $p$ and $s$ preserve the units. Finally, the isomorphism $\sigma$ of Definition~\ref{def_ideal} is the identity map~$\id_A$ and we can conclude that $\xi$ is an ideal action.
\end{proof}

This allows us to state the following.

\begin{theorem}\label{char_algebras}
Let $\bV$ be a unit-closed variety of non-associative algebras and consider the ideally exact context \eqref{eq:situation_alg}. Let
\[
\xi\colon U(B)\flat X\to X
\]
be a relative $U$-action with associated split epimorphism 
\[
\begin{tikzcd}
{A} & {U(B).}
\arrow["{p}", shift left, from=1-1, to=1-2]
\arrow["{s}", shift left, from=1-2, to=1-1]
\end{tikzcd}
\] 
Then $\xi$ is a coherent action if and only if $A$ is a unital algebra and $s(1_B)=1_A$.
\end{theorem}

\begin{proof}
If $\xi$ is a coherent action, then by Proposition~\ref{conv_algebras} $A$ is a unital algebra and $s(1_B)=1_A$. Conversely, if $A$ is a unital algebra and $s(1_B)=1_A$, then 
\[
p(1_A)=p(s(1_B))=(p \circ s)(1_B)=1_B
\]
and the unique morphism $f \colon U(\bF \ltimes X) \to A$ induced by $\id_X$ and $U(\iota_B)$ makes the following a split pullback diagram:
\[\begin{tikzcd}
{U(\bF \ltimes X)} & {U(\bF)} \\
A & {U(B).}
\arrow["{UF(\tau_X)}", shift left, from=1-1, to=1-2]
\arrow["f"', dashed, from=1-1, to=2-1]
\arrow["{UF(\iota_X)}", shift left, from=1-2, to=1-1]
\arrow["{U(\iota_B)}", from=1-2, to=2-2]
\arrow["p", shift left, from=2-1, to=2-2]
\arrow["s", shift left, from=2-2, to=2-1]
\end{tikzcd}\]
Thus, $\xi$ is a coherent action.
\end{proof}

\begin{remark}
The condition $s(1_B)=1_A$ automatically implies that
\[
\begin{tikzcd}
{A} & {U(B)}
\arrow["{p}", shift left, from=1-1, to=1-2]
\arrow["{s}", shift left, from=1-2, to=1-1]
\end{tikzcd}
\] 
is an ideal split epimorphism, i.e., it is a split epimorphism of unital algebras.
\end{remark}

\begin{example}
Let $\bV=\CAssoc$ and consider the internal action $\xi$ associated with the split epimorphism
\[
\begin{tikzcd}
{U(\bF^2)} & {U(\bF)}
\arrow["{U(\pi_1)}", shift left, from=1-1, to=1-2]
\arrow["{U(s)}", shift left, from=1-2, to=1-1]
\end{tikzcd}
\] 
under the equivalence \eqref{equivalence}, where $\bF^2$ is the direct product of two copies of the field $\bF$, $\pi_1(a,b)=a$ and $s(a)=(a,a)$. Then~$\xi$ is a coherent action since $s(1)=(1,1)=1_{\bF^2}$.

If we replace the section $s$ with the canonical inclusion on the first component~$i_1$, then the corresponding internal action is not coherent since $i_1(1)=(1,0) \neq 1_{\bF^2}$.
\end{example}

\begin{example}
Let $\bV=\Assoc$, let $A=\operatorname{UT}_2(\bF)$ be the algebra of $2 \times 2$ upper triangular matrices and let $B$ be the subalgebra of $A$ of matrices of the form
\[
\begin{pmatrix}
a & 0 \\
0 & 0 \\
\end{pmatrix}
\]
with $a \in \bF$. We observe that both $A$ and $B$ are unital algebras with
\[
1_A= \begin{pmatrix}
1 & 0 \\
0 & 1 \\
\end{pmatrix} \neq \begin{pmatrix}
1 & 0 \\
0 & 0 \\
\end{pmatrix} = 1_B,
\]
thus $B$ is not a unital subalgebra of $A$. We consider the split epimorphism 
\begin{equation}\label{alg_counterex}
\begin{tikzcd}
{U(A)} & {U(B)}
\arrow["{U(p)}", shift left, from=1-1, to=1-2]
\arrow["{s}", shift left, from=1-2, to=1-1]
\end{tikzcd}    
\end{equation}
in $\Assoc$ defined by
\[
p \begin{pmatrix}
a & b \\
0 & c \\
\end{pmatrix}=\begin{pmatrix}
a & 0 \\
0 & 0 \\
\end{pmatrix} \quad \text{and} \quad
s \begin{pmatrix}
a & 0 \\
0 & 0 \\
\end{pmatrix}=\begin{pmatrix}
a & 0 \\
0 & 0 \\
\end{pmatrix}.
\]
Since $s(1_B) \neq 1_A$, the internal action associated with the split epimorphism \eqref{alg_counterex} under the equivalence \eqref{equivalence} is not coherent. Indeed, if 
\[
X=\ker U(p) = \bigg\lbrace \begin{pmatrix}
0 & b \\
0 & c \\
\end{pmatrix} \; \bigg\vert \; b,c \in \bF \bigg\rbrace
\]
and $k \colon X \to U(A)$ is the canonical inclusion, then the unique linear map $f \colon U(\bF \ltimes X) \to U(A)$ which makes the following diagram commute
\begin{equation*}
\begin{tikzcd}
X \arrow [r, "\eta_X"] \ar[d, equal]
&U(\bF \ltimes X) \arrow[r, shift left, "UF(\tau_X)"] \ar[d, dashed, "f"']&
U(\bF) \ar[l, shift left, "UF(\iota_X)"]\ar[d, "U(\iota_B)"]\\
X \arrow [r, "k"]
&U(A) \arrow[r, shift left, "U(p)"] &
U(B) \ar[l, shift left, "s"]
\end{tikzcd}
\end{equation*}
is defined by
\[
f\bigg(a, \begin{pmatrix}
0 & b \\
0 & c \\
\end{pmatrix}\bigg)=\begin{pmatrix}
a & b \\
0 & c \\
\end{pmatrix}.
\]
One may easily check that $f$ is not an algebra morphism, since
\[
f\bigg(0, \begin{pmatrix}
0 & 1 \\
0 & 1 \\ 
\end{pmatrix}\bigg) \cdot f\bigg(1, \begin{pmatrix}
0 & 1 \\
0 & 0 \\ 
\end{pmatrix}\bigg)=\begin{pmatrix}
0 & 1 \\
0 & 1 \\
\end{pmatrix} \cdot \begin{pmatrix}
1 & 1 \\
0 & 0 \\
\end{pmatrix} = \begin{pmatrix}
0 & 0 \\
0 & 0 \\
\end{pmatrix},
\]
while
\begin{align*}
f\bigg( \bigg( 0, \begin{pmatrix}
0 & 1 \\
0 & 1 \\ 
\end{pmatrix} \bigg) \cdot \bigg(1, \begin{pmatrix}
0 & 1 \\
0 & 0 \\ 
\end{pmatrix}\bigg) \bigg)=f \bigg( 0, \begin{pmatrix}
0 & 1 \\
0 & 1 \\
\end{pmatrix}\bigg)=\begin{pmatrix}
0 & 1 \\
0 & 1 \\
\end{pmatrix}.
\end{align*}
\end{example}

As an immediate consequence of Theorem~\ref{char_algebras}, we get the following.

\begin{theorem}
Let $\bV$ be a unit-closed variety of non-associative algebras over $\bF$. The ideally exact context
\[
\begin{tikzcd}
{\bV_1} & {\bV}
\arrow[""{name=0, anchor=center, inner sep=0}, "U"', from=1-1, to=1-2]
\arrow[""{name=1, anchor=center, inner sep=0}, "F"', curve={height=14pt}, from=1-2, to=1-1]
\arrow["\dashv"{anchor=center, rotate=-90}, draw=none, from=1, to=0]
\end{tikzcd}
\]
is BAT.
\end{theorem}

\begin{proof}
By Proposition~\ref{conv_algebras}, we have that every coherent action is ideal. It remains to show that every morphism between ideal split epimorphisms is ideal.

Let $B$ be an algebra of $\bV_1$ and consider a morphism 
\[
\begin{tikzcd}
{A_1} & {} & {A_2} \\
& {U(B)}
\arrow["h", from=1-1, to=1-3]
\arrow["{p_1}"', shift right, from=1-1, to=2-2]
\arrow["p_2"', from=1-3, to=2-2]
\arrow["{s_1}"'{pos=0.3}, shift right, from=2-2, to=1-1]
\arrow["s_2"', shift right=2, from=2-2, to=1-3]
\end{tikzcd}
\]
between ideal split epimorphisms over $U(B)$. It follows from Theorem~\ref{char_algebras} that $A_1$ and $A_2$ are unital algebras, with $1_{A_1}=s_1(1_B)$ and $1_{A_2}=s_2(1_B)$. Thus, one has
\[
h(1_{A_1})=h(s_1(1_B))=s_2(1_B)=1_{A_2},
\]
i.e., $h$ is an ideal morphism.
\end{proof}

We conclude this section by noting that if we replace the field $\bF$ with the ring of integers $\bZ$, then both Proposition~\ref{conv_algebras} and Theorem~\ref{char_algebras} hold in the ideally exact context \eqref{eq:situation_rng}, and we may state the following.

\begin{theorem}\label{BAT_Rng}
The ideally exact context
\[\begin{tikzcd}
{\Ring} & {\Rng}
\arrow[""{name=0, anchor=center, inner sep=0}, "U"', from=1-1, to=1-2]
\arrow[""{name=1, anchor=center, inner sep=0}, "F"', curve={height=14pt}, from=1-2, to=1-1]
\arrow["\dashv"{anchor=center, rotate=-90}, draw=none, from=1, to=0]
\end{tikzcd}\]
is BAT. \noproof
\end{theorem}

\subsection{Varieties of hoops, MV-algebras and product algebras}\label{MV}

The algebraic structure now known as a \emph{hoop} was first introduced by B.~Bosbach in~\cite{bosbach1,bosbach2}, where it appeared under the name \emph{complementary semigroups} (\emph{komplementäre Halbgruppen}). The term hoop itself was later coined in an unpublished manuscript by J.~R.~Büchi and T.~M.~Owens~\cite{BuchiOwens}, and has since become standard in the literature on substructural logics and residuated structures. We refer the reader to \cite{hoop} for more details about the structures of hoops.

In this section, we investigate the notion of coherent actions in some varieties of hoops. We aim to prove that the variety of \emph{Wajsberg hoops} and that of \emph{product hoops} define two BAT ideally exact contexts.

\begin{definition}
A \emph{hoop} is an algebra $H=(H, \cdot, \w, 1)$ of type $(2,2,0)$ such that
\begin{enumerate}
\item[(H1)] $(H,\cdot, 1)$ is a commutative monoid;
\item[(H2)] $x \w x=1;$
\item[(H3)] $x \cdot (x \w y)=y \cdot (y \w x)$;
\item[(H4)] $(x \cdot y)\w z= x \w (y \w z)$,
\end{enumerate}
for every $x,y,z \in H$.
\end{definition}

Given two hoops $H,K$, a \emph{hoop homomorphism} is a map $f \colon H \to K$ that preserves both the binary operations $\cdot$ and $\w$, and the constant $1$.

\begin{remark}
Every hoop $H$ is endowed with a partial order $\leq$, which is defined by the following equivalent conditions: for every $x,y \in H$
\begin{enumerate}
\item $x \leq y$;
\item $x \w y=1$;
\item there exists $z \in H$ such that $x=z \cdot y$.
\end{enumerate}
\end{remark}

Given a homomorphism of hoops $f \colon H \to K$, its kernel
\[
F=\ker f =\{ h \in H \mid f(h)=1 \}
\]
is a \emph{filter} of $H$, i.e., a subset $F\subseteq H$ such that $(F, \cdot, 1)$ is a submonoid of $(H, \cdot, 1)$, which is upward closed with respect to the partial order $\leq$ of $H$. Conversely, every filter $F$ of $H$ may be seen as the kernel of the canonical projection $\pi \colon H \to H/F$.

\begin{remark}\cite{rel}
The variety $\Hp$ of hoops is a semi-abelian category.
\end{remark}

As mentioned above, we focus on the study of coherent actions in two relevant subvarieties of the variety $\Hp$: the variety of \emph{Wajsberg hoops} and that of \emph{product hoops}. 

\begin{definition}
A \emph{bounded hoop} is a hoop $H$ with a constant $0$ such that
\begin{itemize}
\item[(H5)] $0 \w x=1$,
\end{itemize}
for any $x \in H$.
\end{definition}

\begin{definition}
A \emph{Wajsberg hoop} is a hoop $H$ such that
\begin{itemize}
\item[(W)] $(x \w y)\w y=(y \w x)\w x$,
\end{itemize}
for every $x,y \in H$.
\end{definition}

It was proved in~\cite{CignoliMundici} that bounded Wajsberg hoops are term equivalent to the class of \emph{MV-algebras}~\cite{chang1, chang2}, which constitutes the equivalent algebraic semantics of \emph{Łukasiewicz logic}~\cite{Luka}.

\begin{definition}\cite{chang1}
An \emph{MV-algebra} is an algebra $A=(A, \oplus, \neg, 0)$ of type $(2,1,0)$ such that 
\begin{itemize}
\item[(MV1)] $(A, \oplus,0)$ is a commutative monoid;
\item[(MV2)] $\neg \neg x=x$;
\item[(MV3)] $x \oplus \neg 0=\neg 0$;
\item[(MV4)] $ \neg (\neg x \oplus y)\oplus y= \neg (\neg y \oplus x)\oplus x$,
\end{itemize}
for every $x,y \in A$.
\end{definition}

\begin{example}
The set $A=[0,1]$ endowed with the operations $x \oplus y=\min \{ x+y,1 \}$, $\neg x=1-x$ and constant $0$ is an MV-algebra. In fuzzy logic~\cite{Fuzzy1, Fuzzy2, hajek}, this algebra is called the \emph{standard MV-algebra}, as it forms the standard real-valued semantics of Łukasiewicz logic.
\end{example}

\begin{remark}
The class of MV-algebras forms an algebraic variety which we denote by $\MV$, whose initial object is the two-element Boolean algebra $L_2=\{ 0,1 \}$, with $1= \neg 0$, while the terminal object is the trivial MV-algebra $\{ 1 \}$. Hence, the variety $\MV$ is not semi-abelian, since it is not pointed, but, as shown in~\cite{rel}, it is protomodular~\cite{caratt}. As a consequence, $\MV$ is ideally exact and the semi-abelian category $(\MV\downarrow L_{2})$ is equivalent to the category $\WH$ of \emph{Wajsberg hoops}.
\end{remark}

\begin{remark}\label{rem_MV}
Given an MV-algebra $A$, one may define the constant $1=\neg 0$ and two binary operations $\odot$ and $\to$ on $A$ as follows:
\[
x \odot y \coloneqq \neg (\neg x \oplus \neg y), \quad x \to y \coloneqq \neg x \oplus y
\]
for every $x,y \in A$. One may check that $(A,\odot, \to, 1)$ is a bounded Wajsberg hoop with $0$ as bottom element. Moreover, $\neg 1=0$ and $\neg x = x \to 0$.

This defines the forgetful functor $U\colon \MV \to \WH$ that forgets the constant~$0$. The left adjoint $M\colon \WH \to \MV$ of $U$, which is called the \emph{MV-closure} in~\cite{MVclos}, maps any Wajsberg hoop $H=(H,\cdot,\w, 1)$ to the MV-algebra 
\[
M(H)=(H \times \{0,1\},\cdot, \w, 0,1),
\]
where $0 \coloneqq (1,0)$, $1 \coloneqq (1,1)$,
\[
(a,i)\cdot (b,j)=\begin{cases}
(a \cdot b,1), \quad &\text{if} \ i=j=1, \\
(a \w b,0), &\text{if} \ i=1, \ j=0,\\
(b \w a,0), & \text{if} \ i=0, \ j=1,\\
((a\w (a \cdot b))\w b,0), &\text{if} \ i=j=0.
\end{cases}
\]
and
\[
(a,i)\w (b,j)=\begin{cases}
(a \w b,1), &\text{if} \ i=j=1, \\
(a \cdot b,0), &\text{if} \ i=1, \ j=0,\\
((a\w (a \cdot b))\w b,1), & \text{if} \ i=0, \ j=1,\\
(b \w a,1), & \text{if} \ i=j=0.
\end{cases}
\]

One may check that the unit $\eta\colon 1_{\WH}\Rightarrow UM$ of the adjunction $M \dashv U$ is cartesian. In fact, for any Wajsberg hoop $H$, the homomorphism
\[
\eta_H \colon H \to UM(H) \colon x \mapsto (x,1)
\]
defines a kernel of 
\begin{align*}
U(M(\tau_H))\colon UM(H)&\to U(L_{2}) \colon (x,i)\mapsto i,
\end{align*}
where $L_2=M(\{ 1 \})$. Thus, by Remark~\ref{cartesian}, we have that
\begin{equation}\label{eq:situation_MV}
\begin{tikzcd}
{\MV} & {\WH}
\arrow[""{name=0, anchor=center, inner sep=0}, "U"', from=1-1, to=1-2]
\arrow[""{name=1, anchor=center, inner sep=0}, "M"', curve={height=14pt}, from=1-2, to=1-1]
\arrow["\dashv"{anchor=center, rotate=-90}, draw=none, from=1, to=0]
\end{tikzcd}
\end{equation}
is, up to an equivalence, the adjunction associated with the unique map $L_2 \to \{1\}$ in~$\MV$.
\end{remark}

Our aim is now to characterise coherent actions in the ideally exact context \eqref{eq:situation_MV}. We first observe that, as for unit-closed varieties of algebras, the converse of Theorem~\ref{thmunital} actually holds also in $\WH$.
 
\begin{proposition}\label{conv_MV}
Consider the ideally exact context \eqref{eq:situation_MV} and let
\[
\xi\colon U(B)\flat X\to X
\]
be a relative $U$-action in $\WH$. If $\xi$ is coherent, then $\xi$ is an ideal action.
\end{proposition}

\begin{proof}
Let $\xi\colon U(B)\flat X\to X$ be a coherent action in $\WH$ and let
\begin{equation*}\label{ideal_split_WH}
\begin{tikzcd}
{A} & {U(B).}
\arrow["{p}", shift left, from=1-1, to=1-2]
\arrow["{s}", shift left, from=1-2, to=1-1]
\end{tikzcd}
\end{equation*}
be a split epimorphism associated with $\xi$ under the equivalence \eqref{equivalence}. By Lemma~\ref{unit}, there exists a morphism $f\colon UM(X)\to A$ in $\WH$ such that the following diagram
\[
\begin{tikzcd}
X & {UM(X)} & {U(L_2)} \\
X & A & {U(B)}
\arrow["{\eta_{X}}", from=1-1, to=1-2]
\arrow[equal, from=1-1, to=2-1]
\arrow["{UM(\tau_X)}", shift left, from=1-2, to=1-3]
\arrow["f"', dashed, from=1-2, to=2-2]
\arrow["{UM(\iota_X)}", shift left, from=1-3, to=1-2]
\arrow["{U(\iota_B)}", from=1-3, to=2-3]
\arrow["k"', from=2-1, to=2-2]
\arrow["p", shift left, from=2-2, to=2-3]
\arrow["s", shift left, from=2-3, to=2-2]
\end{tikzcd}
\]
commutes. For any $a \in A$, we have that $ s(0_{B})\to a \in X$, since 
\[
p(s(0_{B})\to a)=ps(0_{B})\to p(a)=0_{B}\to p(a)=1_B.
\]
It follows that $(s(0_{B})\to a,1) \in UM(X)$. Moreover
\begin{align*}
1_A=f(1_A,1)=f((1_A,0)\to (s(0_{B})\to a,1))&=f(1_A,0)\to f(s(0_{B})\to a,1)\\
&=(f \circ UM(\iota_X))(0)\to (f \circ \eta_{X})(s(0_{B})\to a)\\
&=s(0_{B})\to (s(0_{B})\to a)\\
&=(s(0_{B})\cdot s(0_{B}))\to a\\
&=s(0_{B}\cdot 0_{B})\to a\\
&=s(0_{B})\to a.
\end{align*}
Thus, $s(0_{B})$ is the bottom element of $A$ and $A$ is a bounded Wajsberg hoop. This means that $A$ defines an MV-algebra $A'$, an isomorphism $\sigma$ between $U(A')$ and $A$, and a split epimorphism
\[
\begin{tikzcd}
{A'} & {B}
\arrow["{p'}", shift left, from=1-1, to=1-2]
\arrow["{s'}", shift left, from=1-2, to=1-1]
\end{tikzcd}
\] 
such that $p \circ \sigma=U(p')$ and $\sigma^{-1} \circ s=U(s')$. Thus, the internal action $\xi$ is ideal.
\end{proof}

This allows us to state the following characterisation.

\begin{theorem}\label{car_MV}
Consider the ideally exact context \eqref{eq:situation_MV} and let 
\[
\xi\colon U(B)\flat X\to X
\]
be a relative $U$-action in $\WH$ with associated split epimorphism 
\[\begin{tikzcd}
{A} & {U(B).}
\arrow["{p}", shift left, from=1-1, to=1-2]
\arrow["{s}", shift left, from=1-2, to=1-1]
\end{tikzcd}\] 
Then $\xi$ is a coherent action if and only if $A$ is a bounded Wajsberg hoop with bottom element $s(0_B)$.
\end{theorem}

\begin{proof}
If $\xi$ is coherent, then by Proposition~\ref{conv_MV} $A$ is a bounded Wajsberg hoop with bottom element $s(0_B)$. Conversely, if $A$ is a bounded Wajsberg hoop with bottom element $0_A=s(0_B)$, then $A \cong U(A')$ for some MV-algebra $A'$ and
\[
p(0_A)=p(s(0_B))=(p \circ s)(0_B)=0_B.
\]
Hence, $p=U(p')$, $s=U(s')$ for some split epimorphism
\[
\begin{tikzcd}
{A'} & {B}
\arrow["{p'}", shift left, from=1-1, to=1-2]
\arrow["{s'}", shift left, from=1-2, to=1-1]
\end{tikzcd}
\]
in $\MV$. Thus, the action $\xi$ is ideal, and consequently, it is coherent.
\end{proof}

\begin{remark}
The condition $s(0_B)=0_A$ automatically implies that
\[
\begin{tikzcd}
{A} & {U(B)}
\arrow["{p}", shift left, from=1-1, to=1-2]
\arrow["{s}", shift left, from=1-2, to=1-1]
\end{tikzcd}
\] 
is an ideal split epimorphism, i.e., it is a split epimorphism of bounded Wajsberg hoops.
\end{remark}

\begin{remark}
It follows from Theorem~\ref{car_MV} that there is a morphism 
\[
\Tilde{f} \colon UM(X) \to U(A')
\]
making the following diagram in $\WH$ commutative:
\begin{equation}\label{diag_MV}
\begin{tikzcd}
X & {UM(X)} & {U(L_{2})} \\
X & {U(A')} & {U(B).}
\arrow["{\eta_{X}}", from=1-1, to=1-2]
\arrow[equal, from=1-1, to=2-1]
\arrow["{UM(\tau_X)}", shift left, from=1-2, to=1-3]
\arrow["\Tilde{f}"', dashed, from=1-2, to=2-2]
\arrow["{UM(\iota_X)}", shift left, from=1-3, to=1-2]
\arrow["{U(\iota_B)}", from=1-3, to=2-3]
\arrow["k"', from=2-1, to=2-2]
\arrow["{U(p')}", shift left, from=2-2, to=2-3]
\arrow["{U(s')}", shift left, from=2-3, to=2-2]
\end{tikzcd}
\end{equation}
This morphism arises from the universal property of the unit $\eta$. Indeed, there exists a morphism $f\colon M(X)\to A'$ in $\MV$ such that the following diagram in $\WH$
\[
\begin{tikzcd}
X & {UM(X)}\\
{U(A')}
\arrow["{\eta_{X}}", from=1-1, to=1-2]
\arrow["k"', from=1-1, to=2-1]
\arrow["{U(f)}", dashed, from=1-2, to=2-1]
\end{tikzcd}
\]
is commutative. It follows that $U(f)(x,1)=k(x)$, $U(f)(1_A,0)=0_A$ and
\[
U(f)(x,0)=U(f)((x,1)\to (1_{A},0))=U(f)(x,1)\to U(f)(1_A,0)=k(x) \to 0_A=\neg k(x).
\]
Furthermore, one may check that
\[
(U(f) \circ UM(\iota_X))(i)=U(f)(1_A,i)=i_A=U(s')(i_B)=(U(s') \circ U(\iota_B))(i),
\]
for any $i=0,1$,
\[
(U(p')\circ U(f))(x,1)=U(p')(k(x))=1_B=U(\iota_B)(1)=(U(\iota_B) \circ UM(\tau_X))(x,1)
\]
and
\begin{align*}
(U(p')\circ U(f))(x,0)&=U(p')(\neg k(x))=\neg U(p')(k(x))=\\
&=\neg 1_B=0_B=U(\iota_B)(0)=(U(\iota_B) \circ UM(\tau_X))(x,0).
\end{align*}
Hence, $\Tilde{f}=U(f)$ makes diagram \eqref{diag_MV} commute.
\end{remark}
 
\begin{example}
Let $A$ be an MV-algebra, let $\pi_{1} \colon A\times A \to A$ be the projection on the first component and let $s(a)=(a,a)$ be a section of $\pi_1$. We consider the split epimorphism in $\MV$
\[
\begin{tikzcd}
{A \times A} & A.
\arrow["{\pi_{1}}", shift left, from=1-1, to=1-2]
\arrow["s", shift left, from=1-2, to=1-1]
\end{tikzcd}
\]
Thus, $U(A \times A)$ is a bounded Wajsberg hoop with bottom element $(0_A,0_A)=s(0_A)$, and the internal action $\xi$ induced by the split epimorphism
\[
\begin{tikzcd}
{U(A \times A)} & U(A)
\arrow["{U(\pi_{1})}", shift left, from=1-1, to=1-2]
\arrow["U(s)", shift left, from=1-2, to=1-1]
\end{tikzcd}
\]
under the equivalence \eqref{equivalence}, is coherent.
    
In fact, if $X \coloneqq \ker U(\pi_{1})=\{(1_A,a) \mid a \in A\}$ and $k \colon X \to U(A \times A)$ denotes the canonical inclusion, then there exists a unique morphism $f\colon UM(X)\to U(A \times A)$ in $\WH$ which makes the following diagram 
\[
\begin{tikzcd}
X & {UM(X)} & {U(L_{2})} \\
X & {U(A\times A)} & {U(A)}
\arrow["{\eta_{X}}", from=1-1, to=1-2]
\arrow[equal, from=1-1, to=2-1]
\arrow["{UM(\tau_X)}", shift left, from=1-2, to=1-3]
\arrow["f"', dashed, from=1-2, to=2-2]
\arrow["{UM(\iota_X)}", shift left, from=1-3, to=1-2]
\arrow["{U(\iota_A)}", from=1-3, to=2-3]
\arrow["k"', from=2-1, to=2-2]
\arrow["{U(\pi_{1})}", shift left, from=2-2, to=2-3]
\arrow["{U(s)}", shift left, from=2-3, to=2-2]
\end{tikzcd}
\]
commutative. In particular, for any $a \in A$, one has
\[
f((1_A,a),1)=(1_A,a)
\]
and
\[
f((1_A,a),0)=(1_A,a)\to(0_A,0_A)=(0_A,\neg a).
\]
\end{example}

\begin{example}
Let $A$ be an MV-algebra and consider the split epimorphism
\begin{equation}\label{MV_counterex}
\begin{tikzcd}
{U(A \times A)} & U(A).
\arrow["{U(\pi_{1})}", shift left, from=1-1, to=1-2]
\arrow["s'", shift left, from=1-2, to=1-1]
\end{tikzcd}  
\end{equation}
in $\WH$, where $\pi_{1}(a,b)=a$ and $s'(a)=(a,1_A)$. Since $s'(0)=(0_A,1_A) \neq (0_A,0_A)$, the internal action induced by the split epimorphism \eqref{MV_counterex} under the equivalence \eqref{equivalence} is not coherent. In fact, there does not exist a morphism
\[
f \colon UM(X) \to U(A \times A)
\]
in $\WH$ such that the diagram
\[
\begin{tikzcd}
X & {UM(X)} & {U(L_2)} \\
X & {U(A\times A)} & {U(A)}
\arrow["{\eta_{X}}", from=1-1, to=1-2]
\arrow[equal, from=1-1, to=2-1]
\arrow["{UM(\tau_X)}", shift left, from=1-2, to=1-3]
\arrow["f"', dashed, from=1-2, to=2-2]
\arrow["{UM(\iota_X)}", shift left, from=1-3, to=1-2]
\arrow["{U(\iota_A)}", from=1-3, to=2-3]
\arrow["k"', from=2-1, to=2-2]
\arrow["{U(\pi_{1})}", shift left, from=2-2, to=2-3]
\arrow["{s'}", shift left, from=2-3, to=2-2]
\end{tikzcd}\]
commutes. One may check that the unique map which makes the diagram commutative in $\Set$ is defined by
\[
f(a,1)=(1_A,a), \qquad f(a,0)=(1_A,a) \to (0_A,1_A).
\]
However, $f$ is not a morphism in $\WH$ since 
\[
f(x,0) \to f(y,1) = (1_A,y)
\]
while
\[
f((x,0) \to (y,1))=(1_A,x \oplus y),
\]
for every $x,y \in A$.
\end{example}

As a direct consequence of Theorem~\ref{car_MV}, we get the following.

\begin{theorem}\label{BAT_MV}
The ideally exact context
\[
\begin{tikzcd}
{\MV} & {\WH}
\arrow[""{name=0, anchor=center, inner sep=0}, "U"', from=1-1, to=1-2]
\arrow[""{name=1, anchor=center, inner sep=0}, "M"', curve={height=14pt}, from=1-2, to=1-1]
\arrow["\dashv"{anchor=center, rotate=-90}, draw=none, from=1, to=0]
\end{tikzcd}
\]
is BAT.
\end{theorem}

\begin{proof}
By Proposition~\ref{conv_MV}, we know that every coherent action is ideal. It remains to show that every morphism between ideal split epimorphisms is ideal.

Let $B$ be an MV-algebra and consider a morphism 
\[
\begin{tikzcd}
{A_1} & {} & {A_2} \\
& {U(B)}
\arrow["h", from=1-1, to=1-3]
\arrow["{p_1}"', shift right, from=1-1, to=2-2]
\arrow["p_2"', from=1-3, to=2-2]
\arrow["{s_1}"'{pos=0.3}, shift right, from=2-2, to=1-1]
\arrow["s_2"', shift right=2, from=2-2, to=1-3]
\end{tikzcd}
\]
between ideal split epimorphisms over $U(B)$. It follows from Theorem~\ref{car_MV} that $A_i$ is a bounded Wajsberg hoop with bottom element $0_{A_i}=s_i(0_B)$, for any $i=1,2$, and
\[
h(0_{A_1})=h(s_1(0_B))=s_2(0_B)=0_{A_2}.
\]
Thus, since the class of MV-algebras is term-equivalent to the class of bounded Wajsberg hoops, there exists a morphism
\[
\begin{tikzcd}
{A_1'} & {} & {A_2'} \\
& {B'}
\arrow["h'", from=1-1, to=1-3]
\arrow["{p_1'}"', shift right, from=1-1, to=2-2]
\arrow["p_2'"', from=1-3, to=2-2]
\arrow["{s_1'}"'{pos=0.3}, shift right, from=2-2, to=1-1]
\arrow["s_2'"', shift right=2, from=2-2, to=1-3]
\end{tikzcd}
\]
between split epimorphisms over $B$, such that $U(h')=h$. Hence, $h$ is an ideal morphism.
\end{proof}

Results similar to those of Proposition~\ref{conv_MV}, Theorem~\ref{car_MV} and Theorem~\ref{BAT_MV} can be obtained for \emph{product hoops} and \emph{product algebras}. We start by recalling the following definitions.
 
\begin{definition}\cite{basichoop}
A \emph{basic hoop} is a hoop $H$ such that
\begin{itemize}
\item[(B)] $((x \to y)\to z)\to (((y \to x)\to z)\to z)=1$,
\end{itemize}
for any $x,y,z \in H$.
\end{definition}

\begin{definition}\cite{prodred}
A \emph{product hoop} is a basic hoop $H$ satisfying the identity
\begin{itemize}
\item[(P)] $(y \to z)\vee ((y\to (x \cdot y))\to x)=1$,
\end{itemize}
for any $x,y,z \in H$, where
\[
x \vee y=((x \to y) \to y) \wedge ((y \to x) \to x)
\]
and
\[
x \wedge y=x \cdot (x \to y).
\]
\end{definition}

We denote by $\PH$ the variety of product hoops. Similarly to the case of Wajsberg hoops, it was proved in~\cite{prodred} that bounded product hoops are term equivalent to the class of \emph{product algebras}, which were introduced in~\cite{godo} and constitute the equivalent algebraic semantics of \emph{product logic}~\cite{prodred, CignoliTorrens1}. 

\begin{definition}\cite{hajek}
A \emph{BL-algebra} is an algebra $A=(A, \vee, \wedge, \cdot, \to, 0,1)$ such that
\begin{itemize}
\item[(BL1)] $(A,\vee,\wedge, \cdot, \to, 0, 1)$ is a bounded residuated lattice~\cite{res_lattices};
\item[(BL2)] $x \wedge y=x \cdot (x \to y)$;
\item[(BL3)] $(x \to y)\vee (y \to x)=1$,
\end{itemize}
for any $x,y \in A$.
\end{definition}

\begin{definition}\cite{godo}
A \emph{product algebra} is a BL-algebra $A$ satisfying the identity
\[
\neg x \vee ((x \to x \cdot y) \to y)=1,
\]
for any $x,y \in A$, where $\neg x \coloneqq x \to 0.$
\end{definition}

We denote by $\PA$ the variety of product algebras. Since $\PA$ is protomodular, it is an ideally exact category and it has as initial object the two-element Boolean algebra $L_2$, and as terminal one the trivial product algebra $\{ 1 \}$.

\begin{example}
An example of product algebra is given by the set $A=[0,1]$ endowed with the operations $x \cdot_{\prod} y=xy$ and
\[
x \to_{\prod} y=\begin{cases} 1, \ \ &\text{if} \ x \leq y, \\
\frac{y}{x}, &\text{otherwise.} \end{cases}
\]
In fuzzy logic, this is called the \emph{standard product algebra}, as it forms the standard real-valued semantics of \emph{product logic}.
\end{example}

\begin{remark}\label{rem_prod}
Given a product algebra $A=(A, \vee, \wedge, \cdot, \to, 0,1)$, one may check that $(A,\cdot,\to,1)$ is a bounded product hoop. This defines the forgetful functor $U\colon \PA \to \PH$. In~\cite{prod} the authors provided a description of the left adjoint of such $U$. We now recall the construction that freely adds the constant~$0$ to a product hoop. 

Let $H$ be a product hoop. It may be shown that every element $x \in H$ can be decomposed into Boolean and cancellative components given by the following terms:
\[
b(x)=(x \to x^{2})\to x, \quad c(x)=x \to x^{2}.
\]
Moreover, the set $G(H)=\{ b(x) \mid x \in H\} $ is a \emph{generalised Boolean algebra}~\cite{gen_bool}, $C(H)=\{ c(x) \mid x \in H \}$ is a \emph{cancellative hoop}~\cite{hoop} and the MV-closure $B(H)=M(G(H))$ of $G(H)$ is a Boolean algebra. 
 
Now, let $H^{\bullet} \coloneqq \{ x^{\bullet} \mid x \in H\}$ and let $\sim$ be the equivalence relation on $H \cup H^{\bullet}$ defined by
\[
x \sim x' \quad \text{if and only if} \quad b(x)=b(x') \ \text{and} \ \neg b(x)\vee_{H}c(x)=\neg b(x')\vee_{H} c(x'),
\]
where $\vee_{H} \colon B(H)\times C(H)\to C(H)$ is defined by
\[
b \vee_{H} c=\begin{cases}
b \vee c, \ & \text{if} \ b \in G(H), \\
\neg b \to c, & \text{otherwise.}
\end{cases}
\]
One may check that the set $H \cup H^{\bullet}/\sim$ endowed with the operations
\begin{align*}
&x \cdot y^{\bullet}=((b\to b')\wedge c \cdot c')^{\bullet},\\
&x^{\bullet} \cdot y^{\bullet}=((b \vee b')\wedge c \cdot c')^{\bullet},\\
&x^{\bullet}\to y^{\bullet}=(b'\to b)\wedge(b \vee(c\to c')),\\
&x^{\bullet}\to y=(b \vee b')\wedge(b \vee(c \to c'))\\
&x \to y^{\bullet}=(b \wedge b'\wedge (b \cdot c \to c'))^{\bullet},
\end{align*}
where $b = b(x)$, $c = c(x)$, $b'=b(x')$ and $c'=c(x')$, is a product algebra. Moreover, the free functor $K\colon \PH\to \PA$, which sends a product hoop $H$ to the product algebra $K(H)=H\cup H^{\bullet}/\sim$, is the left adjoint of the forgetful functor $U\colon \PA \to \PH$.

Finally, the unit $\eta \colon 1_{\PH}\Rightarrow UK$ of the adjunction $K\dashv U$ is cartesian. Indeed, for any product hoop $H$, the morphism $\eta_{H}\colon H \to UK(H) \colon x \mapsto x$ is a kernel of
\begin{align*}
UK(\tau_H)\colon UK(H)&\to U(L_{2})=UK(0)\\
x& \mapsto \begin{cases}
1, \ & \text{if} \ x \in H,\\
0, & \text{otherwise.}
\end{cases}
\end{align*}
Thus, by Remark~\ref{cartesian}, we have that
\begin{equation}\label{eq:situation_PR}
\begin{tikzcd}
{\PA} & {\PH}
\arrow[""{name=0, anchor=center, inner sep=0}, "U"', from=1-1, to=1-2]
\arrow[""{name=1, anchor=center, inner sep=0}, "K"', curve={height=14pt}, from=1-2, to=1-1]
\arrow["\dashv"{anchor=center, rotate=-90}, draw=none, from=1, to=0]
\end{tikzcd}
\end{equation}
is, up to an equivalence, the adjunction associated with the unique map $L_2 \to \{ 1 \}$ in $\PA$.
\end{remark}

We are now ready to show that the converse of Theorem~\ref{thmunital} holds for the ideally exact context \eqref{eq:situation_PR}.

\begin{proposition}\label{conv_prod}
Consider the ideally exact context \eqref{eq:situation_PR} and let 
\[
\xi\colon U(B)\flat X\to X
\]
be a relative $U$-action in $\PH$. If $\xi$ is coherent, then $\xi$ is an ideal action.
\end{proposition}
 
\begin{proof}
Let $\xi\colon U(B)\flat X\to X$ be a coherent action in $\PH$ with associated split epimorphism
\begin{equation*}
\begin{tikzcd}\label{ideal_split_pr}
{A} & {U(B).}
\arrow["{p}", shift left, from=1-1, to=1-2]
\arrow["{s}", shift left, from=1-2, to=1-1]
\end{tikzcd}
\end{equation*}
By Lemma~\ref{unit}, there exists a morphism $f\colon UK(X)\to A$ in $\PH$ such that the following diagram
\[\begin{tikzcd}
X & {UK(X)} & {U(L_2)} \\
X & A & {U(B)}
\arrow["{\eta_{X}}", from=1-1, to=1-2]
\arrow[equal, from=1-1, to=2-1]
\arrow["{UK(\tau_X)}", shift left, from=1-2, to=1-3]
\arrow["f"', dashed, from=1-2, to=2-2]
\arrow["{UK(\iota_X)}", shift left, from=1-3, to=1-2]
\arrow["{U(\iota_B)}", from=1-3, to=2-3]
\arrow["k"', from=2-1, to=2-2]
\arrow["p", shift left, from=2-2, to=2-3]
\arrow["s", shift left, from=2-3, to=2-2]
\end{tikzcd}\]
commutes. As in the case of Wajsberg hoops, it is sufficient to prove that $A$ is a bounded product hoop with bottom element $s(0_B)$, i.e., $s(0_{B})\to a=1_A$ for any $a \in A$.

For any $a \in A$, we have
\[
p(s(0_{B})\to a)=ps(0_{B})\to p(a)=0_{B}\to p(a)=1_B.
\]
Thus, $s(0_{B})\to a \in X$ and
\begin{align*}
f(1^{\bullet}\to (s(0_{B})\to a))&=f(1)=1_A.
\end{align*}
Moreover
\begin{align*}
f(1^{\bullet})\to f(s(0_{B})\to a)&=s(0_{B})\to (s(0_{B})\to a)\\
&=s(0_{B})\cdot s(0_{B})\to a\\
&=s(0_{B})\to a,
\end{align*}
i.e., $s(0_{B})\to a=1_A$. Hence, the internal action $\xi$ is ideal.
\end{proof}

This allows us to state the following characterisation for coherent and ideal actions in the ideally exact context \eqref{eq:situation_PR}.

\begin{theorem}\label{car_PR}
Consider the ideally exact context \eqref{eq:situation_PR} and let 
\[
\xi\colon U(B)\flat X\to X
\]
be a relative $U$-action in $\PH$ with associated split epimorphism
\[
\begin{tikzcd}
{A} & {U(B).}
\arrow["{p}", shift left, from=1-1, to=1-2]
\arrow["{s}", shift left, from=1-2, to=1-1]
\end{tikzcd}
\]
Then $\xi$ is a coherent action if and only if $A$ is a bounded product hoop with bottom element~$s(0_B)$. \noproof
\end{theorem}

Again, as a direct consequence of the previous characterisation, we get the following result, whose proof is analogous to that of Theorem~\ref{BAT_MV}.

\begin{theorem}
The ideally exact context
\[
\begin{tikzcd}
{\PA} & {\PH}
\arrow[""{name=0, anchor=center, inner sep=0}, "U"', from=1-1, to=1-2]
\arrow[""{name=1, anchor=center, inner sep=0}, "K"', curve={height=14pt}, from=1-2, to=1-1]
\arrow["\dashv"{anchor=center, rotate=-90}, draw=none, from=1, to=0]
\end{tikzcd}
\]
is BAT. \noproof
\end{theorem}

We end this section by presenting a class of examples of coherent and ideal actions in the variety of product hoops.

\begin{example}\cite{strongsect, Montagna}
Let $A$ be a product algebra. Let 
\[
\operatorname{B}(A)=\{x \in A \mid \neg \neg x=x\}
\]
be the set of \emph{regular elements} of $A$ and let
\[
\operatorname{D}(A)=\{x \in A \mid \neg \neg x=1_A\}
\]
be the set of \emph{dense elements} of $A$. Then $\operatorname{B}(A)$ is the greatest Boolean subalgebra of $A$, $\operatorname{D}(A)$ is a filter of $U(A)$ and one may consider the split extension
\[
\begin{tikzcd}
\operatorname{D}(A) & U(A) & {U(\operatorname{B}(A))}
\arrow["k", from=1-1, to=1-2]
\arrow["p", shift left, from=1-2, to=1-3]
\arrow["s", shift left, from=1-3, to=1-2]
\end{tikzcd}
\]
where $p(a)= \neg \neg a$ (see~\cite[Theorem~1.2 and Lemma~1.4]{CignoliTorrens1} where it is proved that $p$ is a homomorphism), and $k$ and $s$ are the canonical inclusions. Then, by Theorem~\ref{car_PR} the action associated with the split extension above is coherent since $s(0_A)=0_A$. Notice that this split extension has \emph{strong section} (see~\cite{strongsect_pr, strongsect, Rump1, Rump2}).
\end{example}

\subsection{A non-varietal example}\label{sec_set}

We conclude the manuscript by presenting an example that illustrates how the converse of Theorem~\ref{thmunital} holds true beyond the framework of varieties. To achieve this, we consider the ideally exact category $\Set^{\op}$ (see~\cite[Example 3.7]{IdeallyExact}), which has the singleton $1=\{ * \}$ as initial object, and the empty set $\emptyset$ as terminal one.

In this case, the monadic adjunction with cartesian unit of Remark~\ref{cartesian} may be described by
\begin{equation}\label{eq:situation_set}
\begin{tikzcd}
{\Set^{\op}} & {(\Set_*)^{\op},}
\arrow[""{name=0, anchor=center, inner sep=0}, "U"', from=1-1, to=1-2]
\arrow[""{name=1, anchor=center, inner sep=0}, "F"', curve={height=16pt}, from=1-2, to=1-1]
\arrow["\dashv"{anchor=center, rotate=-90}, draw=none, from=1, to=0]
\end{tikzcd}
\end{equation}
where $(\Set_*)^{\op}$ is the dual of the category of pointed sets, and $U$ maps any set $A$ to the pointed set $(1+A,*)$, where $+$ denotes the disjoint union, and any map $f \colon A \to B$~to
\[
1+f \colon (1+A,*) \to (1+B,*),
\]
which is defined by $(1+f)(*)=*$ and $(1+f)(a)=f(a)$, for any $a \in A$. Furthermore, $F(X,*_X)=X$ for any pointed set $(X,*_X)$. We observe that the unit and the counit of the adjunction
\[
\eta_{(X,*_X)} \colon UF(X,*_X)=(1+X,*) \to (X,*_X),
\]
\[
\varepsilon_{A} \colon A \to FU(A)=1+A
\]
are defined by
\[
\eta_{(X,*_X)}(x)=x, \quad \eta_{(X,*_X)}(1)=*_X, \quad \varepsilon_{A}(a)=a.
\]
We aim to show that the ideally exact context \eqref{eq:situation_set} is BAT.

\begin{remark}
Let us observe that the adjunction we just described is nothing but the dual of the one giving rise to the so-called \emph{maybe monad}, widely used in computer science, see~\cite{maybe_mon}.
\end{remark} 

\begin{remark}
One may check that, as in the previous varietal cases, the functor $U$ of diagram \eqref{eq:situation_set} is full on isomorphisms. In fact, an isomorphism $\alpha \colon (1+A,*) \to (1+B,*)$ in $(\Set_*)^{\op}$ is nothing but a bijection between $1+A$ and $1+B$ such that $\alpha(1)=1$. Hence, the restriction $\beta=\alpha_{\downharpoonright_{A}}$ is a bijection between $A$ and $B$ such that $U(\beta)=\alpha$.
\end{remark}

Now, let $(B,*_B)$ and $(X,*_X)$ be pointed sets. A split extension of $(B,*_B)$ by $(X,*_X)$ in the semi-abelian category $(\Set_*)^{\op}$ may be described as a diagram in $\Set_*$
\[\begin{tikzcd}
{(B,*_B)} & {(A,*_{A})} & {(X,*_{X})}
\arrow["p", shift left, from=1-1, to=1-2]
\arrow["s", shift left, from=1-2, to=1-1]
\arrow["k", from=1-2, to=1-3]
\end{tikzcd}\]
where $p$ is a split monomorphism, $s$ is a split epimorphism, $s \circ p=\id_B$, and there is a canonical isomorphism $(X,*_X) \cong (A/p(B), [*_A])$.

\begin{remark}
As shown in~\cite{Deval}, since $k$ is a normal epimorphism, there exists a unique splitting 
\[
\delta \colon X \to A
\]
defined by 
\[
\delta([a])=
\begin{cases}
a, \quad &\text{if }[a] \neq [*_A],\\
*_A, \quad &\text{if }[a]=[*_A].
\end{cases}
\]  
\end{remark}

We aim to prove now that, as in the previous examples, the converse of Theorem~\ref{thmunital} actually holds also in the ideally exact context \eqref{eq:situation_set}.

\begin{proposition}\label{conv_set}
Consider the ideally exact context \eqref{eq:situation_set} and let
\[
\xi \colon U(B)\flat (X,*_{X}) \to (X,*_{X})
\] be a relative $U$-action in $(\Set_{*})^{\op}$. If $\xi$ is coherent, then $\xi$ is an ideal action.
\end{proposition}
\begin{proof}

Let $\xi \colon U(B)\flat (X,*_{X}) \to (X,*_{X})$ be a coherent action in $(\Set_{*})^{\op}$ with associated split monomorphism in $\Set_*$ 
\begin{equation}\label{ideal_split_Set}
\begin{tikzcd}
{(1+B,*)} & {(A,*_{A}).}
\arrow["p", shift left, from=1-1, to=1-2]
\arrow["s", shift left, from=1-2, to=1-1]
\end{tikzcd}    
\end{equation}
Then, by Lemma~\ref{unit} there exists a map $f \colon (A,*_{A}) \to (1+X,*)$ in $\Set_{*}$ such that the following diagram
\[
\begin{tikzcd}
{(1+1, *)} & {(1+X, *)} & {(X,*_X)} \\
{(1+B,*)} & {(A,*_A)} & {(X,*_X)}
\arrow["{1+\iota_X}", shift left, from=1-1, to=1-2]
\arrow["{1+\tau_X}", shift left, from=1-2, to=1-1]
\arrow["\eta_{(X,*_X)}", from=1-2, to=1-3]
\arrow["{1+\tau_B}", from=2-1, to=1-1]
\arrow["p", shift left, from=2-1, to=2-2]
\arrow["f", dashed, from=2-2, to=1-2]
\arrow["s", shift left, from=2-2, to=2-1]
\arrow["k", from=2-2, to=2-3]
\arrow[equal, from=2-3, to=1-3]
\end{tikzcd}
\]
commutes, where $(1+1,*)=UF(1,*)$, $1+\iota_X=1+F(\iota_{(X,*_X)})=UF(\iota_{(X,*_X)})$ and $1+\tau_X=1+F(\tau_{(X,*_X)})=UF(\tau_{(X,*_X)})$.

We observe that for any $a \in A$, $s(a)=*$ if and only if $a=p(*)=*_A$. Indeed, if $a=p(*)$, then $s(a)=s(p(*))=*$. Conversely, if $s(a)=*$, it follows by 
\[
*=((1+\tau_B)\circ s)(a)=((1+\tau_X)\circ f)(a)
\]
that $f(a)=*$. Hence
\[
k(a)=(\eta_{(X,*_X)} \circ f)(a)=\eta_{(X,*_X)}(*)=*_{X}.
\]
As a consequence $a \in \ker k$, i.e., there exists $x\in 1+B$ such that $p(x)=a$. If $x \in B$, then 
\[
f(a)=f(p(x))= ((1+\iota_X)\circ (1+\tau_B))(x)=*_{X} \in X,
\]
which is a contradiction since $f(a)=* \notin X$. Thus, $a=p(*)$.

To conclude the proof, we take $A'=A \setminus \{ *_A \}$, $p'=p_{\downharpoonright_{B}}$ and $s'= s_{\downharpoonright_{A'}}$. One may easily check that $U(p')=p$, $U(s')=s$ and the map
\[
\sigma \colon (A,*_A) \to (1+A',*)
\]
defined by
\[
\sigma(a)=
\begin{cases}
a, \quad &\text{if }a \neq *_A\\
*, \quad &\text{if }a= *_A
\end{cases}
\]  
is an isomorphism in $\Set_*$. Thus, the split monomorphism \eqref{ideal_split_Set} is ideal.
\end{proof}

\begin{remark}
Let
\[\begin{tikzcd}
{(1+B, *)} & {(A,p(1)).}
\arrow["p", shift left, from=1-1, to=1-2]
\arrow["s", shift left, from=1-2, to=1-1]
\end{tikzcd}\]
be a split monomorphism in $\Set_{*}$. Since $(A,p(1))$ is isomorphic to $(1 + A',*)$, where $A'=A \setminus \{ p(1) \}$, and $s_{\downharpoonright_{A'}}$ can be defined if and only if $s^{-1}(*)=\{ *_A \}$, we may state the following characterisation of coherent/ideal actions in the category $(\Set_*)^{\op}$.
\end{remark}

\begin{theorem}\label{char_set}
Consider the ideally exact context \eqref{eq:situation_set} and let
\[
\xi \colon U(B)\flat (X,*_{X}) \to (X,*_{X})
\] be a relative $U$-action in $(\Set_{*})^{\op}$ with associated split monomorphism in $\Set_{*}$
\[\begin{tikzcd}
{(1+B,*)} & {(A,*_{A}).}
\arrow["p", shift left, from=1-1, to=1-2]
\arrow["s", shift left, from=1-2, to=1-1]
\end{tikzcd}\]
Then $\xi$ is a coherent action if and only if $s^{-1}(*)= \{ *_A \}$. \noproof
\end{theorem}

The next example shows that the notion of coherent/ideal action does not trivialise in~$(\Set_*)^{\op}$.

\begin{example}
Consider the pointed sets $(A,*)$ and $(B, *)$, where $A=\{ a, * \}$ and $B=\{ * \}$. Let $p \colon B \hookrightarrow A$ be the inclusion and consider its retraction $s \colon A \to B$ defined by $s(a)=s(*)=*$. The associated action is not coherent since there exists no function $s' \colon \{ * \} \to \emptyset$.
\end{example}

To conclude the manuscript, we aim to prove that the ideally exact context \eqref{eq:situation_set} has a good theory of actions.

\begin{theorem}
The ideally exact context
\[
\begin{tikzcd}
{\Set^{\op}} & {(\Set_*)^{\op},}
\arrow[""{name=0, anchor=center, inner sep=0}, "U"', from=1-1, to=1-2]
\arrow[""{name=1, anchor=center, inner sep=0}, "F"', curve={height=15pt}, from=1-2, to=1-1]
\arrow["\dashv"{anchor=center, rotate=-90}, draw=none, from=1, to=0]
\end{tikzcd}
\]
is BAT.
\end{theorem}

\begin{proof}
By Proposition~\ref{conv_set}, we know that every coherent relative $U$-action is ideal. It remains to show that every morphism between ideal split monomorphisms in $\Set_*$ is ideal.

Let $B$ be a set and consider a morphism
\[
\begin{tikzcd}
{(A_2,*_2)} & {} & {(A_1,*_1)} \\
& {(1+B,*)}
\arrow["h", from=1-1, to=1-3]
\arrow["{s_2}"', shift right, from=1-1, to=2-2]
\arrow["s_1"', from=1-3, to=2-2]
\arrow["{p_2}"'{pos=0.3}, shift right, from=2-2, to=1-1]
\arrow["p_1"', shift right=2, from=2-2, to=1-3]
\end{tikzcd}
\]
between ideal split monomorphisms over $(1+B,*)$.

It follows from Theorem~\ref{char_set} that $s_i^{-1}(*)=*_i$, for $i=1,2$. In addition, if $p_i(b)=*_i$, then $b=s_i(p_i(b)))=s_i(*_i)=*$. This means that $p_i^{-1}(*_i)=*$, for any~$i=1,2$. 

Thus, if $A_i'=A_i \setminus \{ *_i \}$, we have that $p_i'=p_i{_{\downharpoonright_{B}}}$ and $s_i'=s_i{_{\downharpoonright_{A_i'}}}$ define two split monomorphisms in $\Set$
\[
\begin{tikzcd}
{B} & {A_i',}
\arrow["p_i'", shift left, from=1-1, to=1-2]
\arrow["s_i'", shift left, from=1-2, to=1-1]
\end{tikzcd} \quad i=1,2
\]
and it is immediate to check that 
\[
\begin{tikzcd}
{A_2} & {} & {A_1} \\
& {B}
\arrow["h'=h_{\downharpoonright_{A_i'}}", from=1-1, to=1-3]
\arrow["{s_2'}"', shift right, from=1-1, to=2-2]
\arrow["s_1'"', from=1-3, to=2-2]
\arrow["{p_2'}"'{pos=0.3}, shift right, from=2-2, to=1-1]
\arrow["p_1'"', shift right=2, from=2-2, to=1-3]
\end{tikzcd}
\]
is a morphism of split monomorphisms over $B$ such that $U(h')=h$. In fact, one has that $h(a)=*_1$ if and only if $s_2(a)=s_1(h(a))=s_1(*_1)=*$. Hence, by Theorem~\ref{char_set} we get that $a=*_2$, i.e., $h^{-1}(*_1)=*_2$.
\end{proof}

\section*{Acknowledgments}\label{thanks}
This work is supported by the University of Milan, University of Messina, University of Palermo and by the ``National Group for Algebraic and Geometric Structures and their Applications'' (GNSAGA -- INdAM). The first author is supported by the SDF Sustainability Decision Framework Research Project -- MISE decree of 31/12/2021 (MIMIT Dipartimento per le politiche per le imprese -- Direzione generale per gli incentivi alle imprese) -- CUP:~B79J23000530005, COR:~14019279, Lead Partner:~TD Group Italia Srl, Partner:~University of Palermo. He is also a Postdoctoral Researcher of the Fonds de la Recherche Scientifique--FNRS. The second and third authors were supported by the National Recovery and Resilience Plan (NRRP), Mission 4, Component 2, Investment 1.1, Call for tender No.~1409 published on 14/09/2022 by the Italian Ministry of University and Research (MUR), funded by the European Union -- NextGenerationEU -- Project Title Quantum Models for Logic, Computation and Natural Processes (QM4NP) -- CUP:~B53D23030160001 -- Grant Assignment Decree No. 1371 adopted on 01/09/2023 by the Italian Ministry of University and Research (MUR).

\end{document}